\documentclass{amsart}
\usepackage{amsmath}
\usepackage{amssymb}
\usepackage{graphicx}
\usepackage{amscd}

\newtheorem{theorem}{Theorem}
\theoremstyle{plain}

\newtheorem{corollary}{Corollary}

\newtheorem{definition}{Definition}
\newtheorem{example}{Example}

\newtheorem{lemma}{Lemma}
\newtheorem{notation}{Notation}

\newtheorem{proposition}{Proposition}
\newtheorem{remark}{Remark}

\numberwithin{equation}{section}

\input{tcilatex}

\begin{document}
\title[Deformations of Levi flat structures]{Deformations of Levi flat
structures in smooth manifolds}
\author{Paolo de Bartolomeis}
\address{Universit\`{a} di Firenze\\
Dipartimento di Matematica Applicata "G. Sansone" \\
Via di Santa Marta 3 \\
I-50139 Firenze, Italia }
\email{paolo.debartolomeis@unifi.it}
\author{Andrei Iordan}
\address{Institut de Math\'{e}matiques \\
UMR 7586 du CNRS, case 247 \\
Universit\'{e} Pierre et Marie-Curie \\
4 Place Jussieu \\
75252 Paris Cedex 05\\
France}
\email{iordan@math.jussieu.fr}
\date{May, 22, 2012}

\begin{abstract}
We study intrinsic deformations of Levi flat structures on a smooth
manifold. A Levi flat structure on a smooth manifold $L$ is a couple $\left(
\xi ,J\right) $ where $\xi \subset T\left( L\right) $ is an integrable
distribution of codimension $1$ and $J:\xi \rightarrow \xi $ is a bundle
automorphism which defines a complex integrable structure on each leaf. A
deformation of a Levi flat structure $\left( \xi ,J\right) $ is a smooth
family $\left\{ \left( \xi _{t},J_{t}\right) \right\} _{t\in ]-\varepsilon
,\varepsilon \lbrack }$ of Levi flat structures on $L$ such that $\left( \xi
_{0},J_{0}\right) =\left( \xi ,J\right) $. We define a complex whose
cohomology group of order $1$ contains the infinitesimal deformations of a
Levi flat structure. In the case of real analytic Levi flat structures, this
cohomology group is $H^{1}\left( \mathcal{Z}^{\ast }\left( L\right) ,\delta
\right) \times H^{1}\left( \Lambda _{J}^{0,\ast }\left( \xi \right) \otimes
\xi ,\overline{\partial }_{J}\right) $ where $\left( \mathcal{Z}^{\ast
}\left( L\right) ,\delta ,\left\{ \cdot ,\cdot \right\} \right) $ is the
DGLA associated to $\xi $.
\end{abstract}

\maketitle

\section{\protect\bigskip Introduction}

Let $\Omega $ be a domain with $C^{2}$ boundary in $\mathbb{C}^{n}$, $%
n\geqslant 2$, $\Omega =\left\{ z\in U:\rho \left( z\right) <0\right\} $
where $\rho $ is a $C^{2}$ function defined in a neighborhood $U$ of $%
\partial \Omega $ such that $d\rho \neq 0$ on $\partial \Omega $. The Levi
form introduced by E. E.\ Levi in \cite{Levi10} is the Hermitian form 
\begin{equation*}
\mathcal{L}_{\rho }\left( z,w\right) =\dsum\limits_{i,j=1}^{n}\frac{\partial
^{2}\rho }{\partial z_{i}\partial \overline{z}_{j}}\left( z\right) w_{i}%
\overline{w}_{j},\ z\in \partial \Omega ,\ w\in T_{z}^{\mathbb{C}}\left(
\partial \Omega \right)
\end{equation*}%
where $T_{z}^{\mathbb{C}}\left( \partial \Omega \right) =T_{z}\left(
\partial \Omega \right) \cap JT_{z}\left( \partial \Omega \right) $ is the
maximal complex subspace contained in the tangent space $T_{z}\left(
\partial \Omega \right) $ at $z$ to $\partial \Omega $ and $J$ is the
standard complex structure of $\mathbb{C}^{n}$. The semipositivity of the
Levi form characterizes the domains of holomorphy of $\mathbb{C}^{n}$ (\cite%
{Oka53}, \cite{Bremermann54}, \cite{Norguet54}).

A special situation occurs when the Levi form vanishes i.e. $\partial \Omega 
$ is Levi flat. This case is related to a foliation of $\partial \Omega $ by
complex hypersurfaces, as it was firstly remarked by E.\ Cartan in \cite%
{Cartan32}. \ In fact the following theorem is implicit in \cite{Cartan32}:
a real analytic hypersurface in $\mathbb{C}^{n}$ is Levi flat if and only if
it is locally biholomorphic to a real hyperplane in $\mathbb{C}^{n}$. This
result was generalized for smooth hypersurfaces by F. Sommer \cite{Sommer59}%
: a smooth real hypersurface $L$ in a complex manifold $M$ is Levi flat if
and only if the distribution $\xi =TL\cap JTL$ is integrable, where $J$ is
the complex structure of $M$.

These notions have an intrinsic equivalent: a Levi flat structure on a
smooth manifold $L$ is a couple $\left( \xi ,J\right) $ where $\xi \subset
T\left( L\right) $ is an integrable distribution of codimension $1$ and $%
J:\xi \rightarrow \xi $ is a bundle automorphism which defines a complex
integrable structure on each leaf.

It was proved by W. Lickorish in \cite{Lickorish65} (and in an unpublished
paper by S. Novikov and H. Zieschang) that any compact orientable 3-manifold
has a foliation of codimension $1$. J. Wood proved in \cite{Wood69} that any
compact 3-manifold has a transversally orientable foliation of codimension
1. It follows that any compact orientable 3-manifold admits a Levi flat
structure. In upper dimensions the situation is more complicated and it
seems that the problem of the existence of a Levi flat structure on $\mathbb{%
S}^{2n+1}$, $n\geqslant 2$, is still open (see \cite{Meersseman02} and \cite%
{Meersseman11}).

In this paper we study the deformations of Levi flat structures. The theory
of deformations of complex manifolds was intensively studied from the 50s
beginning with the famous results of Kodaira and Spencer \cite%
{KodairaSpencer58} (see for ex. \cite{Kodaira05},\ \cite{Voisin02}). In \cite%
{Nijenhuis66}, Nijenhuis ans Richardson proved that the deformations of
complex structures are given by solutions of the Maurer-Cartan equation in a
graded Lie algebra by using a theory initiated by \ Gerstenhaber \cite%
{Gerstenhaber64} . This theory was developped following ideas of Deligne by
Goldman and Millson \cite{Goldman88} and in more general situations by M.
Kontsevich \cite{Kontsevich97}, \cite{Kontsevich03}.

It is thus interesting to investigate the deformation theory of other
structures involving complex manifolds via the Maurer-Cartan equation in an
adapted DGLA. This was done by de Bartolomeis and Meylan \cite%
{BartolomeisMeylan10} for strictly pseudoconvex CR structures of
hypersurface type on a contact manifold.

In \cite{Bartolomeis10} the autors studied deformations of Levi flat
hypersurfaces $L$ in compact complex manifolds $M$ . In this case the
complex structure on the leaves is induced by the complex structure of $M$.

As the Levi flat hypersurfaces are characterized by the integrability of the
Levi distribution, we studied firstly in \cite{Bartolomeis10} deformations
of integrable distributions of codimension 1 on smooth manifolds. Thus, we
defined a Differential Graded Lie Algebra (DGLA) $\left( \mathcal{Z}^{\ast
}\left( L\right) ,\delta ,\left\{ \cdot ,\cdot \right\} \right) $ associated
to an integrable distribution of codimension 1 such that the deformations of
this distribution are given by solutions of Maurer-Cartan equation in this
algebra. Then we considered a smooth Levi flat hypersurface $L$ in a complex
manifold and we gave a parametrization of families of smooth hypersurfaces
near $L$ such that the Levi flat deformations are given by the solutions of
the Maurer-Cartan equation in the DGLA associated to the Levi foliation.
Here a Levi-flat deformation of $L$ is a smooth application $\Psi :I\times
M\rightarrow M$, where $I$ is an interval in $\mathbb{R}$ containing the
origin, such that $\Psi _{t}=\Psi \left( t,\cdot \right) \in Diff_{0}\left(
M\right) $,\ $L_{t}=\Psi _{t}L$ is a Levi flat hypersurface in $M$ for every 
$t\in I$ \ and $L_{0}=L$. This allowed us to characterize the infinitesimal
deformations of Levi-flat hypersurfaces in a complex manifold.

The purpose of this paper is to study intrinsic deformations of Levi flat
structures on a smooth manifold. Let $L$ be a smooth manifold and $\left(
\xi ,J\right) $ a Levi flat structure on $L$. A deformation of a Levi flat
structure $\left( \xi ,J\right) $ is a smooth family $\left\{ \left( \xi
_{t},J_{t}\right) \right\} _{t\in ]-\varepsilon ,\varepsilon \lbrack }$ of
Levi flat structures on $L$ such that $\left( \xi _{0},J_{0}\right) =\left(
\xi ,J\right) $.

This situation is totally different from the case of deformations of Levi
flat hypersurfaces in complex manifolds, since the complex structure of the
leaves is not fixed.

In the first paragraph we recall the results obtained in \cite{Bartolomeis10}
for the deformations of integrable distribution of codimension 1 and we give
a few proofs which are used in the sequel.

Then we introduce the complex Lie algebras of derivation type with their
Nijenhuis tensor and $\overline{\partial }$-operator. A Levi flat structure $%
\left( \xi ,J\right) $ on a smooth manifold $L$ induces a natural complex
Lie algebra of derivation type on the algebra $\mathcal{H}\left( \xi \right) 
$ of the vector fields on $L$ which are tangent to $\xi $.

We define a $\left( 0,1\right) $-form $H_{J,\gamma ,X}$ associated to a $%
DGLA $ defining couple $\left( \gamma ,X\right) $ whose $1$-cohomology class
in a modified $\overline{\partial }_{J}$-complex $\mathfrak{\overline{%
\mathfrak{\beth }}}_{J,\gamma ,X}:\Lambda ^{0,\ast }\left( \xi \right)
\otimes \xi \rightarrow \Lambda ^{0,\ast }\left( \xi \right) \otimes \xi $
is independent on the choice of $\left( \gamma ,X\right) $. If the
cohomology class of $H_{J,\gamma ,X}$ in this complex vanishes, the Levi
flat structure is called exact. Every real analytic Levi flat hypersurface
in a complex manifold is exact.

To study infinitesimal deformations of a Levi flat structure $\left( \xi
,J\right) $, we define a complex $\mathfrak{Z}^{\ast }\left( L,\xi \right) $
whose cohomology group of order $1$ contains the set of infinitesimal
deformations of $\left( \xi ,J\right) $. If this set reduces to a point and
thus in particular if the first cohomology group vanishes, we say that $%
\left( \xi ,J\right) $ is infinitesimally rigid.

If $\left( \xi ,J\right) $ is exact, then $H^{1}\left( \mathfrak{Z,d}\right)
=H^{1}\left( \mathcal{Z}\left( L\right) ,\delta \right) \times H^{1}\left(
\Lambda _{J}^{0,\ast }\left( \xi \right) \otimes \xi ,\overline{\partial }%
_{J}\right) $.

Some proofs require tedious computations and we tried to make them as easy
as possible to read.

\section{Deformation theory of integrable distribution of codimension 1}

For simplicity, all the objects considered in the sequel will be smooth of
class $C^{\infty }$. For the convenience of the reader we recall in this
paragraph several basic definitions and results from \cite{Bartolomeis10}:

\subsection{DGLA defining couples}

\begin{definition}
A differential graded Lie agebra (DGLA) is a triple $\left( V^{\ast },d,%
\left[ \cdot ,\cdot \right] \right) $ such that:

1) $V^{\ast }=\oplus _{i\in \mathbb{N}}V^{i}$, where $\left( V^{i}\right)
_{i\in \mathbb{N}}$ \ is a family of $\mathbb{C}$-vector spaces and $%
d:V^{\ast }\rightarrow V^{\ast }$is a graded homomorphism such that $d^{2}=0$%
. An element $a\in V^{k}$ is said to be homogeneous of degree $k=\deg a$.

2) $\left[ \cdot ,\cdot \right] :$ $V^{\ast }\times V^{\ast }\rightarrow
V^{\ast }$defines a structure of graded Lie algebra i.e. for homogeneous
elements we have

\begin{equation}
\left[ a,b\right] =-\left( -1\right) ^{\deg a\deg b}\left[ b,a\right]
\label{antisym}
\end{equation}%
and 
\begin{equation}
\left[ a,\left[ b,c\right] \right] =\left[ \left[ a,b\right] ,c\right]
+\left( -1\right) ^{\deg a\deg b}\left[ b,\left[ a,c\right] \right]
\label{Jacobi}
\end{equation}

3) $d$ is compatible with the graded Lie algebra structure i.e. 
\begin{equation}
d\left[ a,b\right] =\left[ da,b\right] +\left( -1\right) ^{\deg a}\left[ a,db%
\right] .  \label{d(.)}
\end{equation}
\end{definition}

\begin{definition}
Let $\left( V^{\ast },d,\left[ \cdot ,\cdot \right] \right) $ be a DGLA and $%
a\in V^{1}$. We say that $a$ verifies the Maurer-Cartan equation in $\left(
V^{\ast },d,\left[ \cdot ,\cdot \right] \right) $ if%
\begin{equation}
da+\frac{1}{2}\left[ a,a\right] =0.  \label{MC}
\end{equation}
\end{definition}

\begin{definition}
Let $L$ be a $C^{\infty }$ manifold and $\xi \subset T\left( L\right) $ an
integrable distribution of codimension $1$. A couple $\left( \gamma
,X\right) $ where $\gamma \in \wedge ^{1}\left( L\right) $ and $X$ is a
vector field on $L$ such that $Ker~\gamma =\xi $ and $\gamma \left( X\right)
=1$ will be called a DGLA defining couple.
\end{definition}

\begin{lemma}
\label{Forms=DGLA}Let $L$ be a $C^{\infty }$ manifold and $X$ a vector field
on $L$. We denote by $\Lambda ^{k}\left( L\right) $ the $k$-forms on $L$ and 
$\Lambda ^{\ast }\left( L\right) =\oplus _{k\in \mathbb{N}}\Lambda
^{k}\left( L\right) $. For $\alpha ,\beta \in \Lambda ^{\ast }\left(
L\right) $, set%
\begin{equation}
\left\{ \alpha ,\beta \right\} =\mathcal{L}_{X}\alpha \wedge \beta -\alpha
\wedge \mathcal{L}_{X}\beta  \label{Lie braket 1}
\end{equation}%
where $\mathcal{L}_{X}$ is the Lie derivative. Then $\left( \Lambda ^{\ast
}\left( L\right) ,d,\left\{ \cdot ,\cdot \right\} \right) $ is a DGLA.
\end{lemma}

\begin{lemma}
\label{Frobenius} Let $L$ be a $C^{\infty }$ manifold and $\xi \subset
T\left( L\right) $ a distribution of codimension $1$. Let $\left( \gamma
,X\right) $ be a DGLA defining couple. Then the following are equivalent:

i) $\xi $ is integrable;

ii) There exists $\alpha \in \wedge ^{1}\left( L\right) $ such that $d\gamma
=\alpha \wedge \gamma $;

iii) $d\gamma \wedge \gamma =0$;

iv) $d\gamma =-\iota _{X}d\gamma \wedge \gamma $;

v) $\gamma $ satisfies the Maurer-Cartan equation (\ref{MC}) in $\left(
\Lambda ^{\ast }\left( L\right) ,d,\left\{ \cdot ,\cdot \right\} \right) $,
where $\left\{ \cdot ,\cdot \right\} $ is defined in (\ref{Lie braket 1}).
\end{lemma}

\begin{corollary}
\label{Forms+delta=DGLA}Let $L$ be a $C^{\infty }$ manifold and $\xi \subset
T\left( L\right) $ an integrable distribution of codimension $1$. Let $%
\left( \gamma ,X\right) $ be a DGLA defining couple. Set 
\begin{equation*}
\delta =d_{\gamma }=d+\left\{ \gamma ,\cdot \right\}
\end{equation*}%
where $\left\{ \cdot ,\cdot \right\} $ is defined in (\ref{Lie braket 1}).
Then , $\left( \Lambda ^{\ast }\left( L\right) ,\delta ,\left\{ \cdot ,\cdot
\right\} \right) $ is a DGLA.
\end{corollary}

\begin{corollary}
\label{Z* subalgebra}Under the hypothesis of Corollary \ref{Forms+delta=DGLA}%
, we set 
\begin{equation*}
\mathcal{Z}^{\ast }\left( L\right) =\left\{ \alpha \in \Lambda ^{\ast
}\left( L\right) :\ \iota _{X}\alpha =0\right\} .
\end{equation*}%
Then $\left( \mathcal{Z}^{\ast }\left( L\right) ,\delta ,\left\{ \cdot
,\cdot \right\} \right) $ is a sub-DGLA of $\left( \Lambda ^{\ast }\left(
L\right) ,\delta ,\left\{ \cdot ,\cdot \right\} \right) $.
\end{corollary}

\begin{remark}
Let $\alpha ,\beta \in \mathcal{Z}^{\ast }\left( L\right) $ and $\left(
\gamma ,X\right) $ a DGLA defining couple. Then%
\begin{equation}
\left\{ \alpha ,\beta \right\} =\left( \iota _{X}d+d\iota _{X}\right) \alpha
\wedge \beta -\alpha \wedge \left( \iota _{X}d+d\iota _{X}\right) \beta
=\iota _{X}d\alpha \wedge \beta -\alpha \wedge \iota _{X}d\beta
\label{alfa beta}
\end{equation}%
and%
\begin{equation}
\left\{ \gamma ,\alpha \right\} =\left( \iota _{X}d+d\iota _{X}\right)
\gamma \wedge \alpha -\gamma \wedge \left( \iota _{X}d+d\iota _{X}\right)
\alpha =\iota _{X}d\gamma \wedge \alpha -\gamma \wedge \iota _{X}d\alpha .
\label{gama alfa}
\end{equation}
\end{remark}

\ \ \ \ \ \ \ \ \ \ \ \ \ \ \ \ \ \ \ \ \ \ \ \ \ \ \ \ \ \ \ \ \ \ \ \ \ \
\ \ \newline

Let $L$ be a $C^{\infty }$ manifold and $\xi \subset T\left( L\right) $ an
integrable distribution of codimension $1$. We fix a $DGLA$ defining couple $%
\left( \gamma ,X\right) $ and we consider the DGLA $\left( \mathcal{Z}^{\ast
}\left( L\right) ,\delta ,\left\{ \cdot ,\cdot \right\} \right) $ previously
defined.

\begin{lemma}
\label{Ker(gama+alfa) integrable}Let $\alpha \in \mathcal{Z}^{1}\left(
L\right) $. The following are equivalent:

i) The distribution $\xi _{\alpha }=\ker $~$\left( \gamma +\alpha \right) $
is integrable.

ii) $\alpha $ satisfies the Maurer-Cartan equation (\ref{MC}) in \ $\left( 
\mathcal{Z}^{\ast }\left( L\right) ,\delta ,\left\{ \cdot ,\cdot \right\}
\right) $.
\end{lemma}

\subsection{Group action}

\begin{definition}
\label{Def group action}Let $\mathcal{G=}Diff\left( L\right) $ be the group
of diffeomorphisms of $L$. Let $\mathcal{U}_{0}$ be a neighborhood of the
identity in $\mathcal{G}$ and $\mathcal{V}_{0}$ be a neighborhood of $0$ in $%
\mathcal{Z}^{1}\left( L\right) $ such that $\Phi ^{\ast }\left( \gamma
+\alpha \right) \left( X\right) \neq 0$ for every $\Phi \in \mathcal{U}_{0}$
and every $\alpha \in \mathcal{V}_{0}$. We define 
\begin{equation}
\left( \Phi ,\alpha \right) \in \mathcal{U}_{0}\times \mathcal{V}_{0}\subset 
\mathcal{G}\times \mathcal{Z}^{1}\left( L\right) \rightarrow \mathcal{Z}%
^{1}\left( L\right) \ni \chi \left( \Phi \right) \left( \alpha \right)
=\left( \Phi ^{\ast }\left( \gamma +\alpha \right) \left( \mathfrak{X}%
\right) \right) ^{-1}\Phi ^{\ast }\left( \gamma +\alpha \right) -\gamma .
\label{chi(fi)}
\end{equation}
\end{definition}

\begin{remark}
\label{group action}In this way, $\xi _{\chi \left( \Phi \right) \left(
\alpha \right) }=\Phi ^{\ast }\xi _{\alpha }$. This means that $\xi _{\alpha
}$ is integrable if and only if $\xi _{\chi \left( \Phi \right) \left(
\alpha \right) }$ is integrable. By Lemma \ref{Ker(gama+alfa) integrable} we
deduce that $\alpha $ satisfies the Maurer-Cartan equation (\ref{MC}) in the
DGLA\ $\left( \mathcal{Z}^{\ast }\left( L\right) ,\delta ,\left\{ \cdot
,\cdot \right\} \right) $ if and only if $\chi \left( \Phi \right) \left(
\alpha \right) $ does.
\end{remark}

So we obtain the following

\begin{proposition}
\label{moduli 1 foliations}Set%
\begin{equation*}
\mathfrak{MC}_{\delta }\left( L\right) =\left\{ \alpha \in \mathcal{Z}%
^{1}\left( L\right) :\ \delta a+\frac{1}{2}\left\{ \alpha ,\alpha \right\}
=0\right\} .
\end{equation*}%
Then the moduli space of deformations of integrable distributions of
codimension $1$ is 
\begin{equation*}
\mathfrak{MC}_{\delta }\left( L\right) /\thicksim _{\mathcal{G}}
\end{equation*}%
where $\alpha \thicksim _{\mathcal{G}}\beta $ if there exists $\Phi \in 
\mathcal{G}$ such that $\beta =\chi \left( \Phi \right) \left( \alpha
\right) $.
\end{proposition}

\begin{lemma}
\label{d/dt(chi)=delta} Let $Y$ be a vector field on $L$ and $\Phi ^{Y}$ the
flow of $Y$. Then%
\begin{equation*}
\frac{d\chi \left( \Phi _{t}^{Y}\right) }{dt}_{\left\vert t=0\right. }\left(
0\right) =-\delta \left( \iota _{Y}\gamma \right) .
\end{equation*}
\end{lemma}

\begin{proof}
We have 
\begin{eqnarray*}
\frac{d\chi \left( \Phi _{t}^{Y}\right) }{dt}_{\left\vert t=0\right. }\left(
0\right) &=&\frac{d\left( \left( \left( \left( \Phi _{t}^{Y}\right)
^{-1}\right) ^{\ast }\left( \gamma \right) \left( X\right) \right)
^{-1}\left( \left( \Phi _{t}^{Y}\right) ^{-1}\right) ^{\ast }\left( \gamma
\right) -\gamma \right) }{dt}_{\left\vert t=0\right. } \\
&=&\left( \left( \Phi _{t}^{Y}\right) ^{-1}\right) ^{\ast }\left( \gamma
\right) \frac{d\left( \left( \left( \Phi _{t}^{Y}\right) ^{-1}\right) ^{\ast
}\left( \gamma \right) \left( X\right) ^{-1}\right) }{dt}_{\left\vert
t=0\right. }\left( 0\right) \\
&&+\left( \left( \Phi _{t}^{Y}\right) ^{-1}\right) ^{\ast }\left( \gamma
\right) \left( X\right) ^{-1}\frac{d\left( \left( \Phi _{t}^{Y}\right)
^{-1}\right) ^{\ast }}{dt}_{\left\vert t=0\right. } \\
&=&\frac{d\left( \left( \left( \left( \Phi _{t}^{Y}\right) ^{-1}\right)
^{\ast }\left( \gamma \right) \left( X\right) \right) ^{-1}\right) }{dt}%
_{\left\vert t=0\right. }\gamma +\frac{d\left( \left( \left( \Phi
_{t}^{Y}\right) ^{-1}\right) ^{\ast }\left( \gamma \right) \right) }{dt}%
_{\left\vert t=0\right. } \\
&=&\mathcal{L}_{Y}\left( \gamma \right) \left( X\right) \gamma -\mathcal{L}%
_{Y}\gamma \\
&=&\left( d\iota _{Y}\gamma \right) \left( X\right) \gamma +\iota
_{Y}d\gamma \left( X\right) \gamma -d\iota _{Y}\gamma -\iota _{Y}d\gamma .
\end{eqnarray*}%
By Lemma \ref{Frobenius} iv)%
\begin{eqnarray*}
\iota _{Y}d\gamma &=&-\iota _{Y}\left( \iota _{X}d\gamma \wedge \gamma
\right) =-\left( \iota _{Y}\left( \iota _{X}d\gamma \right) \right) \gamma
+\left( \iota _{Y}\gamma \right) \iota _{X}d\gamma \\
&=&-\left( d\gamma \left( X,Y\right) \right) \gamma +\left( \iota _{Y}\gamma
\right) \iota _{X}d\gamma ,
\end{eqnarray*}%
so%
\begin{eqnarray}
\frac{d\chi \left( \Phi _{t}^{Y}\right) }{dt}_{\left\vert t=0\right. }\left(
0\right) &=&\left( d\iota _{Y}\gamma \right) \left( X\right) \gamma -d\gamma
\left( Y,X\right) \gamma -d\iota _{Y}\gamma  \notag \\
&&+\left( d\gamma \left( X,Y\right) \right) \gamma -\left( \iota _{Y}\gamma
\right) \iota _{X}d\gamma  \notag \\
&=&\left( \iota _{X}d\iota _{Y}\gamma \right) \gamma -d\iota _{Y}\gamma
-\left( \iota _{Y}\gamma \right) \iota _{X}d\gamma .  \label{A}
\end{eqnarray}%
Since 
\begin{equation*}
\mathcal{L}_{X}\gamma =d\iota _{X}\gamma +\iota _{X}d\gamma =\iota
_{X}d\gamma
\end{equation*}%
it follows that%
\begin{eqnarray}
\delta \iota _{Y}\gamma &=&d\iota _{Y}\gamma +\left\{ \gamma ,\iota
_{Y}\gamma \right\} =d\iota _{Y}\gamma +\mathcal{L}_{X}\gamma \wedge \iota
_{Y}\gamma -\gamma \wedge \mathcal{L}_{X}\iota _{Y}\gamma  \label{B} \\
&=&d\iota _{Y}\gamma +\left( \iota _{Y}\gamma \right) \iota _{X}d\gamma
-X\left( \iota _{Y}\gamma \right) \gamma .  \notag
\end{eqnarray}%
From (\ref{A}) and (\ref{B}) we obtain 
\begin{equation*}
\frac{d\chi \left( \Phi _{t}^{Y}\right) }{dt}_{\left\vert t=0\right. }\left(
0\right) =-\delta \iota _{Y}\gamma .
\end{equation*}
\end{proof}

\subsection{Infinitesimal deformations}

\begin{definition}
A $\ \mathfrak{MC}_{\delta }\left( L\right) $-valued curve through the
origin is a smooth mapping $\lambda :\left[ -a,a\right] \rightarrow 
\mathfrak{MC}_{\delta }\left( L\right) $, $a>0$, such that $\lambda \left(
0\right) =0$. We say that $\alpha $ is the tangent vector at the origin of
the $\mathfrak{MC}_{\delta }\left( L\right) $-valued curve $\lambda $
through the origin to $\mathfrak{MC}_{\delta }\left( L\right) $ if $\alpha =%
\underset{t\rightarrow 0}{\lim }\frac{\lambda \left( t\right) }{t}=\frac{%
d\lambda }{dt}_{\left\vert t=0\right. }$.
\end{definition}

\begin{proposition}
\label{Tangent Inclus Cohomologie}Let $\alpha $ be the tangent vector at the
origin of a $\mathfrak{MC}_{\delta }\left( L\right) $-valued curve through
the origin $\lambda $, $Y$ a vector field on $L$ and $\Phi ^{Y}$ the flow of 
$Y$. Set $\mu \left( t\right) =\chi \left( \Phi _{t}^{Y}\right) \left(
\lambda \left( t\right) \right) $., Then:

i) $\delta \alpha =0$.

ii)The tangent vector $\beta $ at the origin of \ the $\mathfrak{MC}_{\delta
}\left( L\right) $-valued curve $\mu $ is%
\begin{equation*}
\beta =\alpha -\delta \iota _{Y}\gamma
\end{equation*}
\end{proposition}

\begin{proof}
i) By Lemma \ref{Ker(gama+alfa) integrable} $\lambda \left( t\right) $
verifies the Maurer Cartan equation for every $t$. Since $\lambda \left(
t\right) =\alpha t+o\left( t\right) $, we have $\delta \alpha =0$.

ii) 
\begin{equation}
\beta =\frac{d\mu }{dt}_{\left\vert t=0\right. }=\frac{d}{dt}\chi \left(
\Phi _{t}^{Y}\left( \lambda \left( t\right) \right) \right) _{\left\vert
t=0\right. }=\frac{d\chi \left( \Phi _{t}^{Y}\right) }{dt}_{\left\vert
t=0\right. }\left( 0\right) +\alpha .  \label{beta-beta' debut}
\end{equation}%
The Proposition \ref{Tangent Inclus Cohomologie} follows now by Lemma \ref%
{d/dt(chi)=delta}.
\end{proof}

The Proposition \ref{Tangent Inclus Cohomologie} justifies the following
definition:

\begin{definition}
The infinitesimal deformations of $\xi $ is the collection of cohomology
classes in $H^{1}\left( \mathcal{Z}\left( L\right) ,\delta \right) $ of the
tangent vectors at $0$ to $\mathfrak{MC}_{\delta }\left( L\right) $-valued
curves. We denote by $T_{\left[ 0\right] }\left( \mathfrak{MC}_{\delta
}\left( L\right) /\thicksim _{\mathcal{G}}\right) $ the set of infinitesimal
deformations of $\xi $.
\end{definition}

\subsection{$d_{b}$ operator}

\begin{remark}
\label{iso Z(L) Lambda(L)}We denote $\Lambda ^{\ast }\left( \xi \right)
=\bigoplus_{p\in \mathbb{N}}\Lambda ^{p}\xi ^{\ast }$. There exists a
natural isomorphism $\Theta :\Lambda ^{\ast }\left( \xi \right) \rightarrow 
\mathcal{Z}^{\ast }\left( L\right) $: for $\alpha \in \Lambda ^{1}\xi ^{\ast
}$ set $\Theta \left( \alpha \right) \left( X\right) =0$, $\Theta \left(
\alpha \right) \left( Y\right) =\alpha \left( Y\right) $ if $Y\in \xi $ and
extend by linearity. Let $d_{b}:\Lambda ^{\ast }\left( \xi \right)
\rightarrow \Lambda ^{\ast }\left( \xi \right) $ be the differential along
the leaves of $\xi $. By using this isomorphism we consider $d_{b}:\mathcal{Z%
}^{\ast }\left( L\right) \rightarrow \mathcal{Z}^{\ast }\left( L\right) $
and for every $\alpha \in \mathcal{Z}^{\ast }\left( L\right) $ we have%
\begin{equation}
d_{b}\alpha =\iota _{X}\left( \gamma \wedge d\alpha \right) =d\alpha -\gamma
\wedge \iota _{X}d\alpha .  \label{db}
\end{equation}%
Indeed let $\alpha \in \Lambda ^{p}\xi ^{\ast }$ and $X_{1},\cdot \cdot
\cdot ,X_{p+1}\in \xi $. Since $\gamma \left( X_{j}\right) =0$, $j=1,\cdot
\cdot \cdot ,p+1$ and $\gamma \left( X\right) =1$, we have 
\begin{equation*}
\iota _{X}\left( \gamma \wedge d\alpha \right) \left( X_{1},\cdot \cdot
\cdot ,X_{p+1}\right) =\left( \gamma \wedge d\alpha \right) \left(
X,X_{1},\cdot \cdot \cdot ,X_{p+1}\right) =d\alpha \left( X_{1},\cdot \cdot
\cdot ,X_{p+1}\right) .
\end{equation*}
\end{remark}

\begin{lemma}
\label{i(X) d(gama ) db closed}The form $\iota _{X}d\gamma $ is $d_{b}$%
-closed.
\end{lemma}

\section{Complex Lie algebras of derivation type}

\begin{notation}
Let $V,W$ real vector spaces, $J_{V}$ a complex structure on $V$ , $J_{W}$ a
complex structure on $W$. Then 
\begin{equation*}
\Lambda ^{0,p}V^{\ast }\otimes W=\left\{ \lambda \in \Lambda ^{p}V^{\ast
}\otimes W:\ \lambda \left( v_{1},\cdot \cdot \cdot ,J_{V}v_{k},\cdot \cdot
\cdot ,v_{p}\right) =-J_{W}\ \lambda \left( v_{1},\cdot \cdot \cdot
,v_{k},\cdot \cdot \cdot ,v_{p}\right) \right\} .
\end{equation*}
\end{notation}

\begin{definition}
Let $A$ be a $\mathbb{R}$-algebra with unit and $\mathfrak{g}$ a $A$-module
of derivations of $A$. A structure of complex $A$-Lie algebra of derivation
type is a triple $\left( \mathfrak{g},\left[ \cdot ,\cdot \right] ,J\right) $%
, where $\left( \mathfrak{g},\left[ \cdot ,\cdot \right] \right) $\ is a Lie
algebra verifying%
\begin{equation*}
\left[ aV,W\right] =a\left[ V,W\right] -\left( Wa\right) V
\end{equation*}%
for every $a\in A$, $V,W\in \mathfrak{g}$ and $J$ is a complex structure on $%
\mathfrak{g}$ which is $A$-linear.
\end{definition}

\begin{definition}
Let $\left( \mathfrak{g},\left[ \cdot ,\cdot \right] ,J\right) $ be a
complex $A$-Lie algebra of derivation type. We define the Nijenhuis tensor $%
N=N_{J,\left[ \cdot ,\cdot \right] }$ and \ $\overline{\partial }=\overline{%
\partial }_{J,\left[ \cdot ,\cdot \right] }$ by 
\begin{equation*}
N\left( V,W\right) =\left[ JV,JW\right] -\left[ V,W\right] -J\left[ JV,W%
\right] -J\left[ V,JW\right] ,\ V,W\in \mathfrak{g}
\end{equation*}%
and%
\begin{equation*}
\left( \overline{\partial }W\right) \left( V\right) =\frac{1}{2}\left( \left[
V,W\right] +J\left[ JV,W\right] \right) +\frac{1}{4}N\left( V,W\right) ,\
V,W\in \mathfrak{g}.
\end{equation*}
\end{definition}

\begin{remark}
It is easy to see that $N$ is $A$-bilinear and $N\left( JV,W\right) =N\left(
V,JW\right) =-JN\left( V,W\right) $, so $N\in \Lambda ^{0,2}\mathfrak{g}%
^{\ast }\otimes \mathfrak{g}$. Indeed%
\begin{eqnarray*}
N\left( V,JW\right) &=&\left[ JV,-W\right] -\left[ V,JW\right] -J\left[ JV,JW%
\right] -J\left[ V,-W\right] \\
&=&-\left[ V,JW\right] -\left[ JV,W\right] +J\left[ V,W\right] -J\left[ JV,JW%
\right] =N\left( JV,W\right)
\end{eqnarray*}%
and 
\begin{equation*}
N\left( JV,W\right) =-J\left( \left[ JV,JW\right] -\left[ V,W\right] -J\left[
JV,W\right] -J\left[ V,JW\right] \right) =-JN\left( V,W\right) .
\end{equation*}
\end{remark}

\begin{lemma}
Let $\left( \mathfrak{g},\left[ \cdot ,\cdot \right] ,J\right) $ be a
complex $A$-Lie algebra of derivation type. Then:

i) $\overline{\partial }:\Lambda ^{0,0}\mathfrak{g}^{\ast }\otimes \mathfrak{%
g}\rightarrow \Lambda ^{0,1}\mathfrak{g}^{\ast }\otimes \mathfrak{g}$;

ii) $\overline{\partial }J-J\overline{\partial }=0$;

iii) $\overline{\partial }\left( aW\right) \left( V\right) =\left( \overline{%
\partial }a\right) \left( V\right) W+a\left( \overline{\partial }W\right)
\left( V\right) $ where%
\begin{equation*}
\left( \overline{\partial }a\right) \left( V\right) W=\frac{1}{2}\left(
\left( Va\right) W+J\left( Va\right) JW\right) .
\end{equation*}
\end{lemma}

\begin{proof}
i)%
\begin{eqnarray*}
\left( \overline{\partial }W\right) \left( JV\right) &=&\frac{1}{2}\left( %
\left[ JV,W\right] +J\left[ -V,W\right] \right) +\frac{1}{4}N\left(
JV,W\right) \\
&=&-J\left( \frac{1}{2}\left( J\left[ JV,W\right] +\left[ V,W\right] \right)
+\frac{1}{4}N\left( JV,W\right) \right) =-J\left( \overline{\partial }%
W\right) \left( V\right) .
\end{eqnarray*}

ii) Since%
\begin{equation*}
N\left( V,JW\right) -J\left[ V,W\right] +J\left[ JV,JW\right] +\left[ JV,W%
\right] +\left[ V,JW\right] =0,
\end{equation*}%
we have%
\begin{eqnarray*}
2\left( \overline{\partial }JW\right) \left( V\right) &=&\left[ V,JW\right]
+J\left[ JV,JW\right] +\frac{1}{2}N\left( V,JW\right) \\
&=&J\left[ V,W\right] -\left[ JV,W\right] -\frac{1}{2}N\left( V,JW\right)
=J\left( \left( \overline{\partial }W\right) \left( V\right) \right) .
\end{eqnarray*}

iii)%
\begin{eqnarray*}
\overline{\partial }\left( aW\right) \left( V\right) &=&\frac{1}{2}\left( %
\left[ V,aW\right] +J\left[ JV,aW\right] \right) +\frac{1}{4}N\left(
V,aW\right) \\
&=&\frac{1}{2}\left( -a\left[ W,V\right] +\left( Va\right) W+J\left( -a\left[
W,JV\right] +\left( JVa\right) W\right) \right) +\frac{1}{4}aN\left(
V,W\right) \\
&=&\frac{1}{2}\left( \left( Va\right) W+\left( JVa\right) JW\right) +a\left( %
\left[ V,W\right] +J\left[ JV,W\right] +\frac{1}{4}N\left( V,W\right) \right)
\\
&=&\left( \overline{\partial }a\right) \left( V\right) W+a\left( \overline{%
\partial }W\right) \left( V\right) .
\end{eqnarray*}
\end{proof}

\begin{lemma}
\label{rel complex structures}Let $V$ be a real vector space and $J,%
\widetilde{J}$ complex structures on $J$ such that $\det \left( J+\widetilde{%
J}\right) \neq 0$. There exists a unique $S\in End_{\mathbb{R}}\left(
V\right) $ such that $\widetilde{J}=\left( I+S\right) J\left( I+S\right)
^{-1}$, $SJ+JS=0$.
\end{lemma}

\begin{proof}
Let $S\in End_{\mathbb{R}}\left( V\right) $ such that$\ \widetilde{J}=\left(
I+S\right) J\left( I+S\right) ^{-1},$ $SJ+JS=0$. Then $\widetilde{J}\left(
I+S\right) =\left( I+S\right) J$, so$\ \widetilde{J}+\widetilde{J}%
S=J+SJ=J-JS $ and it follows that $S=\left( J-\widetilde{J}\right) \left( J+%
\widetilde{J}\right) ^{-1}$.

Conversely, define $S=\left( J-\widetilde{J}\right) \left( J+\widetilde{J}%
\right) ^{-1}$. Since 
\begin{eqnarray*}
SJ+JS &=&\left( J-\widetilde{J}\right) \left( J+\widetilde{J}\right)
^{-1}J+J\left( J-\widetilde{J}\right) \left( J+\widetilde{J}\right) ^{-1} \\
&=&\left( J-\widetilde{J}\right) \left( I-\widetilde{J}J\right) ^{-1}+\left(
-I-J\widetilde{J}\right) \left( J+\widetilde{J}\right) ^{-1} \\
&=&\left( J-\widetilde{J}\right) \left( I-\widetilde{J}J\right) ^{-1}+\left( 
\widetilde{J}-J\right) \widetilde{J}\left( J+\widetilde{J}\right) ^{-1} \\
&=&\left( J-\widetilde{J}\right) \left( I-\widetilde{J}J\right) ^{-1}+\left( 
\widetilde{J}-J\right) \left( -\widetilde{J}J+I\right) ^{-1}=0,
\end{eqnarray*}%
it follows that $\widetilde{J}=\left( I+S\right) J\left( I+S\right) ^{-1}$.
\end{proof}

\begin{proposition}
\label{N Ntilda}Let $\left( \mathfrak{g},\left[ \cdot ,\cdot \right]
,J\right) $, $\left( \mathfrak{g},\left[ \cdot ,\cdot \right] ,\widetilde{J}%
\right) $ be complex $A$-Lie algebras of derivation type such that $\det
\left( J+\widetilde{J}\right) \neq 0$. Let $S\in End_{\mathbb{R}}\left( 
\mathfrak{g}\right) $ such that $\widetilde{J}=\left( I+S\right) J\left(
I+S\right) ^{-1}$, $SJ+JS=0$. Then%
\begin{eqnarray}
N_{\widetilde{J}}\left( \left( I+S\right) V,\left( I+S\right) W\right)
&=&\left( I-S\right) ^{-1}\left( N_{J}\left( V,W\right) +S\left( N_{J}\left(
V,W\right) -N_{J}\left( SV,SW\right) \right) \right)  \notag \\
&&-4\left( I-S\right) ^{-1}\left( \overline{\partial }_{J}S+\frac{1}{2}\left[
S,S\right] \right) \left( V,W\right)  \label{Rel N Ntild}
\end{eqnarray}%
where%
\begin{equation*}
\overline{\partial }_{J}S\left( V,W\right) =\overline{\partial }_{J}\left(
SW\right) \left( V\right) -\overline{\partial }_{J}\left( SV\right) \left(
W\right) -\frac{1}{2}S\left( \left[ V,W\right] -\left[ JV,JW\right] \right)
\end{equation*}%
and%
\begin{eqnarray*}
\left[ S,S\right] \left( V,W\right) &=&\left[ SV,SW\right] -\left[ JSV,JSW%
\right] -S\left( \left[ SV,W\right] +\left[ V,SW\right] +J\left[ V,JSW\right]
+J\left[ JSV,W\right] \right) \\
&&-\frac{1}{2}\left( SN_{J}\left( SV,W\right) +SN_{J}\left( V,SW\right)
-N_{J}\left( SV,SW\right) \right) .
\end{eqnarray*}
\end{proposition}

\begin{proof}
Denote 
\begin{equation*}
\widetilde{E}\left( V,W\right) =N_{\widetilde{J}}\left( \left( I+S\right)
V,\left( I+S\right) W\right) .
\end{equation*}%
We have 
\begin{eqnarray*}
\widetilde{E}\left( V,W\right) &=&\left[ \widetilde{J}\left( I+S\right) V,%
\widetilde{J}\left( I+S\right) W\right] -\left[ \left( I+S\right) V,\left(
I+S\right) W\right] \\
&&-\widetilde{J}\left[ \widetilde{J}\left( I+S\right) V,\left( I+S\right) W%
\right] -\widetilde{J}\left[ \left( I+S\right) V,\widetilde{J}\left(
I+S\right) W\right] \\
&=&\left[ \left( I+S\right) JV,\left( I+S\right) JW\right] -\left[ \left(
I+S\right) V,\left( I+S\right) W\right] \\
&&-\left( I+S\right) J\left( I+S\right) ^{-1}\left[ \left( I+S\right)
JV,\left( I+S\right) W\right] \\
&&-\left( I+S\right) J\left( I+S\right) ^{-1}\left[ \left( I+S\right)
V,\left( I+S\right) JW\right] \\
&=&\left[ JV,JW\right] +\left[ SJV,SJW\right] +\left[ JV,SJW\right] +\left[
SJV,JW\right] \\
&&-\left[ V,W\right] -\left[ SV,SW\right] -\left[ V,SW\right] -\left[ SV,W%
\right] \\
&&-\left( I+S\right) J\left( I+S\right) ^{-1}\left( \left[ JV,W\right] +%
\left[ SJV,SW\right] +\left[ JV,SW\right] +\left[ SJV,W\right] \right) \\
&&-\left( I+S\right) J\left( I+S\right) ^{-1}\left( \left[ V,JW\right] +%
\left[ SV,SJW\right] +\left[ V,SJW\right] +\left[ SV,JW\right] \right) .
\end{eqnarray*}%
But%
\begin{eqnarray*}
\left( I+S\right) J\left( I+S\right) ^{-1} &=&\left( I-S\right) ^{-1}\left(
\left( I-S\right) \left( I+S\right) J\left( I+S\right) ^{-1}\right) \\
&=&\left( I-S\right) ^{-1}\left( I-S^{2}\right) J\left( I+S\right) ^{-1} \\
&=&\left( I-S\right) ^{-1}\left( I+S\right) \left( I-S\right) J\left(
I+S\right) ^{-1} \\
&=&\left( I-S\right) ^{-1}\left( I+S\right) \left( J-SJ\right) \left(
I+S\right) ^{-1} \\
&=&\left( I-S\right) ^{-1}\left( I+S\right) \left( J+JS\right) \left(
I+S\right) ^{-1} \\
&=&\left( I-S\right) ^{-1}\left( I+S\right) J
\end{eqnarray*}%
so%
\begin{eqnarray}
\widetilde{E}\left( V,W\right) &=&\left[ JV,JW\right] +\left[ JSV,JSW\right]
-\left[ JV,JSW\right] -\left[ JSV,JW\right]  \notag \\
&&-\left[ V,W\right] -\left[ SV,SW\right] -\left[ V,SW\right] -\left[ SV,W%
\right]  \notag \\
&&-\left( I-S\right) ^{-1}\left( I+S\right) J\left( \left[ JV,W\right] -%
\left[ JSV,SW\right] +\left[ JV,SW\right] -\left[ JSV,W\right] \right) 
\notag \\
&&-\left( I-S\right) ^{-1}\left( I+S\right) J\left( \left[ V,JW\right] -%
\left[ SV,JSW\right] -\left[ V,JSW\right] +\left[ SV,JW\right] \right)
\label{0}
\end{eqnarray}%
We set%
\begin{equation*}
E\left( V,W\right) =\left( I-S\right) \widetilde{E}\left( V,W\right)
\end{equation*}%
and from (\ref{0}) we have%
\begin{eqnarray}
E\left( V,W\right) &=&\left[ JV,JW\right] +\left[ JSV,JSW\right] -\left[
JV,JSW\right] -\left[ JSV,JW\right]  \label{1A} \\
&&-S\left( \left[ JV,JW\right] +\left[ JSV,JSW\right] -\left[ JV,JSW\right] -%
\left[ JSV,JW\right] \right)  \label{2} \\
&&-\left[ V,W\right] -\left[ SV,SW\right] -\left[ V,SW\right] -\left[ SV,W%
\right]  \label{3} \\
&&+S\left( \left[ V,W\right] +\left[ SV,SW\right] +\left[ V,SW\right] +\left[
SV,W\right] \right)  \label{4} \\
&&-J\left( \left[ JV,W\right] -\left[ JSV,SW\right] +\left[ JV,SW\right] -%
\left[ JSV,W\right] \right)  \label{5} \\
&&-SJ\left( \left[ JV,W\right] -\left[ JSV,SW\right] +\left[ JV,SW\right] -%
\left[ JSV,W\right] \right)  \label{6} \\
&&-J\left( \left[ V,JW\right] -\left[ SV,JSW\right] -\left[ V,JSW\right] +%
\left[ SV,JW\right] \right)  \label{7} \\
&&-SJ\left( \left[ V,JW\right] -\left[ SV,JSW\right] -\left[ V,JSW\right] +%
\left[ SV,JW\right] \right) .  \label{8}
\end{eqnarray}

By adding the first terms, (respectively second terms, third terms and forth
terms) in (\ref{1A}), (\ref{3}), (\ref{5}) and (\ref{7}) and then in (\ref{2}%
), (\ref{4}), (\ref{6}) and (\ref{8}) we obtain the following form of $%
E\left( V,W\right) $: 
\begin{eqnarray}
&&\left[ JV,JW\right] -\left[ V,W\right] -J\left[ JV,W\right] -J\left[ V,JW%
\right]  \label{11} \\
&&+\left[ JSV,JSW\right] -\left[ SV,SW\right] +J\left[ JSV,SW\right] +J\left[
SV,JSW\right]  \label{12} \\
&&-\left[ JV,JSW\right] \mathbf{-}\left[ V,SW\right] -J\left[ JV,SW\right] +J%
\left[ V,JSW\right]  \label{13} \\
&&-\left[ JSV,JW\right] -\left[ SV,W\right] +J\left[ JSV,W\right] -J\left[
SV,JW\right]  \label{14} \\
&&+S\left( -\left[ JV,JW\right] +\left[ V,W\right] -J\left[ JV,W\right] -J%
\left[ V,JW\right] \right)  \label{15} \\
&&+S\left( -\left[ JSV,JSW\right] +\left[ SV,SW\right] +J\left[ JSV,SW\right]
+J\left[ SV,JSW\right] \right)  \label{16} \\
&&+S\left( \left[ JV,JSW\right] +\left[ V,SW\right] -J\left[ JV,SW\right] +J%
\left[ V,JSW\right] \right)  \label{17} \\
&&+S\left( \left[ JSV,JW\right] +\left[ SV,W\right] +J\left[ JSV,W\right] -J%
\left[ SV,JW\right] \right)  \label{18}
\end{eqnarray}

Now 
\begin{eqnarray*}
\left( \ref{11}\right) &=&N_{J}\left( V,W\right) \\
\left( \ref{12}\right) &=&N_{J}\left( SV,SW\right) +2\left( J\left[ JSV,SW%
\right] +J\left[ SV,JSW\right] \right) \\
\left( \ref{13}\right) &=&-N_{J}\left( V,SW\right) -2\left( \left[ V,SW%
\right] +J\left[ JV,SW\right] \right) \\
\left( \ref{14}\right) &=&-N_{J}\left( SV,W\right) -2\left( \left[ SV,W%
\right] +J\left[ SV,JW\right] \right) \\
\left( \ref{15}\right) &=&S\left( -\left[ JV,JW\right] +\left[ V,W\right] -J%
\left[ JV,W\right] -J\left[ V,JW\right] \right) \\
\left( \ref{16}\right) &=&-SN_{J}\left( SV,SW\right) \\
\left( \ref{17}\right) &=&SN_{J}\left( V,SW\right) +2S\left( \left[ V,SW%
\right] +J\left[ V,JSW\right] \right) \\
\left( \ref{18}\right) &=&SN_{J}\left( SV,W\right) +2S\left( \left[ SV,W%
\right] +J\left[ JSV,W\right] \right)
\end{eqnarray*}%
so%
\begin{eqnarray}
E\left( V,W\right) &=&N_{J}\left( V,W\right) -SN_{J}\left( SV,SW\right) -2 
\left[ V,SW\right]  \notag \\
&&-2J\left[ JV,SW\right] -2\left[ SV,W\right] -2J\left[ SV,JW\right]  \notag
\\
&&+S\left[ V,W\right] -S\left[ JV,JW\right] -N_{J}\left( V,SW\right)
-N_{J}\left( SV,W\right)  \notag \\
&&+2S\left( \left[ V,SW\right] +J\left[ V,JSW\right] +\left[ SV,W\right] +J%
\left[ JSV,W\right] \right)  \label{E} \\
&&+SN_{J}\left( SV,W\right) +SN_{J}\left( V,SW\right) +N_{J}\left(
SV,SW\right)  \notag \\
&&+2\left( J\left[ JSV,SW\right] +J\left[ SV,JSW\right] \right) +S\left( -J%
\left[ JV,W\right] -J\left[ V,JW\right] \right)  \notag
\end{eqnarray}%
Since%
\begin{eqnarray*}
\overline{\partial }_{J}S\left( V,W\right) &=&\left( \overline{\partial }%
_{J}SW\right) \left( V\right) -\left( \overline{\partial }_{J}SV\right)
\left( W\right) -\frac{1}{2}S\left( \left[ V,W\right] -\left[ JV,JW\right]
\right) \\
&=&\frac{1}{2}\left( \left[ V,SW\right] +J\left[ JV,SW\right] -\left[ W,SV%
\right] -J\left[ JW,SV\right] \right) -\frac{1}{2}\left( S\left[ V,W\right]
-S\left[ JV,JW\right] \right) \\
&&+\frac{1}{4}\left( N_{J}\left( V,SW\right) -N_{J}\left( W,SV\right) \right)
\end{eqnarray*}%
and%
\begin{equation*}
2\left( J\left[ JSV,SW\right] +J\left[ SV,JSW\right] \right) =-2N_{J}\left(
SV,SW\right) +2\left[ JSV,JSW\right] -2\left[ SV,SW\right]
\end{equation*}%
(\ref{E}) becomes%
\begin{eqnarray*}
E\left( V,W\right) &=&N_{J}\left( V,W\right) -SN_{J}\left( SV,SW\right) \\
&&-2\left[ V,SW\right] -2J\left[ JV,SW\right] +2\left[ WS,V\right] +2J\left[
JW,SV\right] \\
&&+2\left( S\left[ V,W\right] -S\left[ JV,JW\right] \right) -\left( S\left[
V,W\right] -S\left[ JV,JW\right] \right) \\
&&-N_{J}\left( V,SW\right) +N_{J}\left( W,SV\right) \\
&&+2S\left( \left[ V,SW\right] +J\left[ V,JSW\right] +\left[ SV,W\right] +J%
\left[ JSV,W\right] \right) \\
&&+SN_{J}\left( SV,W\right) +SN_{J}\left( V,SW\right) +N_{J}\left(
SV,SW\right) \\
&&-2N_{J}\left( SV,SW\right) +2\left[ JSV,JSW\right] -2\left[ SV,SW\right] \\
&&+S\left( -J\left[ JV,W\right] -J\left[ V,JW\right] \right) \\
&=&N_{J}\left( V,W\right) -SN_{J}\left( SV,SW\right) -4\overline{\partial }%
_{J}S\left( V,W\right) -\left( S\left[ V,W\right] -S\left[ JV,JW\right]
\right) \\
&&+2S\left( \left[ V,SW\right] +J\left[ V,JSW\right] +\left[ SV,W\right] +J%
\left[ JSV,W\right] \right) \\
&&+2\left[ JSV,JSW\right] -2\left[ SV,SW\right] +S\left( -J\left[ JV,W\right]
-J\left[ V,JW\right] \right) \\
&&+SN_{J}\left( SV,W\right) +SN_{J}\left( V,SW\right) -N_{J}\left(
SV,SW\right) \\
&=&N_{J}\left( V,W\right) -SN_{J}\left( SV,SW\right) -4\overline{\partial }%
_{J}S\left( V,W\right) \\
&&+2S\left( \left[ V,SW\right] +J\left[ V,JSW\right] +\left[ SV,W\right] +J%
\left[ JSV,W\right] \right) \\
&&+S\left( \left[ JV,JW\right] -\left[ V,W\right] -J\left[ JV,W\right] -J%
\left[ V,JW\right] \right) \\
&&+2\left[ JSV,JSW\right] -2\left[ SV,SW\right] +SN_{J}\left( V,SW\right)
+SN_{J}\left( SV,W\right) -N_{J}\left( SV,SW\right) \\
&=&N_{J}\left( V,W\right) +SN_{J}\left( V,W\right) -SN_{J}\left(
SV,SW\right) -4\overline{\partial }_{J}S\left( V,W\right) \\
&&-2\left[ SV,SW\right] +2\left[ JSV,JSW\right] +2S\left( \left[ SV,W\right]
+J\left[ JSV,W\right] +\left[ V,SW\right] +J\left[ V,JSW\right] \right) \\
&&+SN_{J}\left( VS,W\right) +SN_{J}\left( V,SW\right) -N_{J}\left(
SV,SW\right) \\
&=&N_{J}\left( V,W\right) +SN_{J}\left( V,W\right) -SN_{J}\left(
SV,SW\right) -4\overline{\partial }_{J}S\left( V,W\right) -2\left[ S,S\right]
\left( V,W\right)
\end{eqnarray*}%
which is equivalent to (\ref{Rel N Ntild}).
\end{proof}

From the Proposition \ref{N Ntilda} we obtain:

\begin{corollary}
\label{N Jtilda=0}Let $\left( \mathfrak{g},\left[ \cdot ,\cdot \right]
,J\right) $, $\left( \mathfrak{g},\left[ \cdot ,\cdot \right] ,\widetilde{J}%
\right) $ be complex $A$-Lie algebras of derivation type such that $\det
\left( J+\widetilde{J}\right) \neq 0$. Let $S\in End_{\mathbb{R}}\left( 
\mathfrak{g}\right) $ such that $\widetilde{J}=\left( I+S\right) J\left(
I+S\right) ^{-1}$, $SJ+JS=0$. Then $N_{\widetilde{J}}=0$ if and only if%
\begin{equation*}
\overline{\partial }_{J}S+\frac{1}{2}\left[ \left[ S,S\right] \right] =\frac{%
1}{4}N_{J}
\end{equation*}%
where $\left[ \left[ S,S\right] \right] =\left[ S,S\right] -\frac{1}{2}%
S\left( N_{J}-N_{J}\left( S,S\right) \right) $ and $N_{J}\left( S,S\right)
\left( V,W\right) =N_{J}\left( SV,SW\right) $.
\end{corollary}

\section{Levi flat stuctures}

\begin{definition}
Let $L$ be a smooth manifold. A Levi flat structure on $L$ is a couple $%
\left( \xi ,J\right) $ where $\xi \subset T\left( L\right) $ is an
integrable distribution of codimension $1$ and $J:\xi \rightarrow \xi $
defines a complex integrable structure on each leaf.
\end{definition}

\begin{example}
By W. Lickorish \cite{Lickorish65} (and by an unpublished work of S. Novikov
and H. Zieschang, 1965), any compact orientable 3-manifold has a foliation
of codimension 1. By J. Wood \cite{Wood69} any compact 3-manifold has a
transversaly orientable foliation of codimension 1. It follows that that all
compact orientable $3$-manifolds admit a Levi flat structure.
\end{example}

\begin{notation}
From now on we consider the following setting:

- $L$ a smooth manifold;

- $\left( \xi ,J\right) $ a Levi flat structure on $L$;

- $\left( \gamma ,X\right) $ a $DGLA$ defining couple for $\xi $;

- $\alpha \in \mathcal{Z}^{1}\left( L\right) $ such that $\xi _{\alpha
}=\ker \left( \gamma +\alpha \right) $ is integrable;

- $\mathcal{H}\left( \xi \right) $ the algebra of vector fields on $L$ which
are tangent to $\xi $;

- $\left[ \cdot ,\cdot \right] $ the Lie bracket;

- $\Lambda _{J}^{p,q}\left( \xi \right) $ the $p+q$-forms $\alpha $ on $L$
such that the restriction of $\alpha $ on each leaf $\mathcal{L}$ of the
Levi foliation endowed with the complex structure $J$ is a $\left(
p,q\right) $-form on $\mathcal{L}$ and $\Lambda ^{k}\left( \xi \right)
=\oplus _{p+q=k}\Lambda _{J}^{p,q}\left( \xi \right) $.
\end{notation}

\begin{definition}
For $V,W\in T\left( L\right) $, define $\omega _{\alpha }\left( V\right)
=V-\alpha \left( V\right) X$ and 
\begin{equation*}
\left[ V,W\right] _{\alpha }=\omega _{\alpha }^{-1}\left[ \omega _{\alpha
}\left( V\right) ,\omega _{\alpha }\left( W\right) \right] .
\end{equation*}
\end{definition}

\begin{lemma}
1) $\omega _{\alpha }\in EndT\left( L\right) $ and $\omega _{\alpha }\left(
\xi \right) =\xi _{\alpha }$.

2) $\left( \left( \mathcal{H}\left( \xi \right) \right) _{\alpha },\left[
\cdot ,\cdot \right] _{\alpha },J\right) $ is a complex $C^{\infty }\left(
L\right) $-Lie algebra of derivation type, where $\left( \mathcal{H}\left(
\xi \right) \right) _{\alpha }$ is the $C^{\infty }\left( L\right) $-algebra 
$\mathcal{H}\left( \xi \right) $ endowed with the derivation $\left\langle
V,f\right\rangle _{\alpha }=\omega _{\alpha }\left( V\right) \left( f\right) 
$.
\end{lemma}

\begin{proof}
It is obvious that $\omega _{\alpha }^{-1}\left( V\right) =V+\alpha \left(
V\right) X$ , $\left( \gamma +\alpha \right) \left( V-\alpha \left( V\right)
X\right) =0$ for $V\in \xi $ and $\gamma \left( V+\alpha \left( V\right)
X\right) =0$ for $V\in \ker \left( \gamma +\alpha \right) $.

Let $V,W\in \mathcal{H}\left( \xi \right) $.%
\begin{eqnarray*}
\left[ \omega _{\alpha }\left( V\right) ,\omega _{\alpha }\left( W\right) %
\right] &=&\left[ V-\alpha \left( V\right) X,W-\alpha \left( W\right) X%
\right] \\
&&\left[ V,W\right] -\left[ \alpha \left( V\right) X,W\right] +\left[ \alpha
\left( W\right) X,V\right] +\left[ \alpha \left( V\right) X,\alpha \left(
W\right) X\right] \\
&=&\left[ V,W\right] -\alpha \left( V\right) \left[ X,W\right] +W\left(
\alpha \left( V\right) \right) X+\alpha \left( W\right) \left[ X,V\right]
-V\left( \alpha \left( W\right) \right) X \\
&&+\left[ \alpha \left( V\right) X,\alpha \left( W\right) X\right]
\end{eqnarray*}%
\begin{eqnarray*}
\left[ \alpha \left( V\right) X,\alpha \left( W\right) X\right] &=&\alpha
\left( V\right) \left[ X,\alpha \left( W\right) X\right] -\left( \alpha
\left( W\right) X\right) \left( \alpha \left( V\right) \right) X \\
&=&-\alpha \left( V\right) \left[ \alpha \left( W\right) X,X\right] -\alpha
\left( W\right) \left( X\left( \alpha \left( V\right) \right) \right) X \\
&=&\alpha \left( V\right) X\left( \alpha \left( W\right) \right) X-\alpha
\left( W\right) \left( X\left( \alpha \left( V\right) \right) \right) X
\end{eqnarray*}%
So%
\begin{eqnarray*}
\left[ \omega _{\alpha }\left( V\right) ,\omega _{\alpha }\left( W\right) %
\right] &=&\left[ V,W\right] -\alpha \left( V\right) \left[ X,W\right]
+W\left( \alpha \left( V\right) \right) X+\alpha \left( W\right) \left[ X,V%
\right] -V\left( \alpha \left( W\right) \right) X \\
&&+\alpha \left( V\right) X\left( \alpha \left( W\right) \right) X-\alpha
\left( W\right) \left( X\left( \alpha \left( V\right) \right) \right) X
\end{eqnarray*}%
\begin{eqnarray*}
\left[ \omega _{\alpha }\left( aV\right) ,\omega _{\alpha }\left( W\right) %
\right] &=&a\left[ V,W\right] -W\left( a\right) V-a\alpha \left( V\right) %
\left[ X,W\right] +W\left( a\alpha \left( V\right) \right) X \\
&&-\alpha \left( W\right) \left[ aV,X\right] -aV\left( \alpha \left(
W\right) \right) X+a\alpha \left( V\right) X\left( \alpha \left( W\right)
\right) X-\alpha \left( W\right) \left( X\left( a\alpha \left( V\right)
\right) \right) X \\
&=&a\left[ V,W\right] -W\left( a\right) V-a\alpha \left( V\right) \left[ X,W%
\right] +W\left( a\right) \alpha \left( V\right) X+aW\left( \alpha \left(
V\right) \right) X \\
&&-a\alpha \left( W\right) \left[ V,X\right] +\alpha \left( W\right) X\left(
a\right) V-aV\left( \alpha \left( W\right) \right) X+a\alpha \left( V\right)
X\left( \alpha \left( W\right) \right) X \\
&&-\alpha \left( W\right) \left( X\left( a\right) \alpha \left( V\right)
\right) X-a\alpha \left( W\right) \left( X\left( \alpha \left( V\right)
\right) \right) X
\end{eqnarray*}%
\begin{eqnarray*}
\left[ aV,W\right] _{\alpha } &=&\omega _{\alpha }^{-1}\left[ \omega
_{\alpha }\left( aV\right) ,\omega _{\alpha }\left( W\right) \right] \\
&=&\left[ \omega _{\alpha }\left( aV\right) ,\omega _{\alpha }\left(
W\right) \right] +\alpha \left( \left[ \omega _{\alpha }\left( aV\right)
,\omega _{\alpha }\left( W\right) \right] \right) X
\end{eqnarray*}%
Since%
\begin{eqnarray*}
\alpha \left( \left[ \omega _{\alpha }\left( aV\right) ,\omega _{\alpha
}\left( W\right) \right] \right) &=&a\alpha \left( \left[ V,W\right] \right)
-W\left( a\right) \alpha \left( V\right) -a\alpha \left( V\right) \alpha
\left( \left[ X,W\right] \right) \\
&&-a\alpha \left( W\right) \alpha \left( \left[ V,X\right] \right) +\alpha
\left( W\right) X\left( a\right) \alpha \left( V\right)
\end{eqnarray*}%
we have%
\begin{eqnarray*}
\left[ aV,W\right] _{\alpha } &=&a\left[ V,W\right] -W\left( a\right)
V-a\alpha \left( V\right) \left[ X,W\right] +W\left( a\right) \alpha \left(
V\right) X+aW\left( \alpha \left( V\right) \right) X \\
&&-a\alpha \left( W\right) \left[ V,X\right] +\alpha \left( W\right) X\left(
a\right) V-aV\left( \alpha \left( W\right) \right) X+a\alpha \left( V\right)
X\left( \alpha \left( W\right) \right) X \\
&&-\alpha \left( W\right) \left( X\left( a\right) \alpha \left( V\right)
\right) X-a\alpha \left( W\right) \left( X\left( \alpha \left( V\right)
\right) \right) X \\
&&+a\alpha \left( \left[ V,W\right] \right) X-W\left( a\right) \alpha \left(
V\right) X-a\alpha \left( V\right) \alpha \left( \left[ X,W\right] \right) X
\\
&&-a\alpha \left( W\right) \alpha \left( \left[ V,X\right] \right) X+\alpha
\left( W\right) X\left( a\right) \alpha \left( V\right) X
\end{eqnarray*}%
For $a=1$%
\begin{eqnarray}
\left[ V,W\right] _{\alpha } &=&\left[ V,W\right] -\alpha \left( V\right) %
\left[ X,W\right] +W\left( \alpha \left( V\right) \right) X  \notag \\
&&-\alpha \left( W\right) \left[ V,X\right] -V\left( \alpha \left( W\right)
\right) X+\alpha \left( V\right) X\left( \alpha \left( W\right) \right) X 
\notag \\
&&-\alpha \left( W\right) \left( X\left( \alpha \left( V\right) \right)
\right) X  \notag \\
&&+\alpha \left( \left[ V,W\right] \right) X-\alpha \left( V\right) \alpha
\left( \left[ X,W\right] \right) X  \label{[V,W]alfa} \\
&&-\alpha \left( W\right) \alpha \left( \left[ V,X\right] \right) X  \notag
\end{eqnarray}

So%
\begin{eqnarray*}
\left[ aV,W\right] _{\alpha } &=&a\left[ V,W\right] _{\alpha }-W\left(
a\right) V+W\left( a\right) \alpha \left( V\right) X \\
&&+\alpha \left( W\right) X\left( a\right) V \\
&&-\alpha \left( W\right) \left( X\left( a\right) \alpha \left( V\right)
\right) X \\
&&-W\left( a\right) \alpha \left( V\right) X \\
&&+\alpha \left( W\right) X\left( a\right) \alpha \left( V\right) X \\
&=&a\left[ V,W\right] _{\alpha }-W\left( a\right) V+\alpha \left( W\right)
X\left( a\right) V \\
&=&a\left[ V,W\right] _{\alpha }-\left( W-\alpha \left( W\right) X\right)
\left( a\right) V \\
&=&a\left[ V,W\right] _{\alpha }-\omega _{\alpha }\left( W\right) \left(
a\right) V=a\left[ V,W\right] _{\alpha }-\left\langle W,a\right\rangle
_{\alpha }V.
\end{eqnarray*}
\end{proof}

\section{The $\left( 0,1\right) $-form associated to a DGLA defining couple}

\begin{notation}
i) $\overline{\partial }=\overline{\partial _{J}}:\Lambda _{J}^{0,p}\left(
\xi \right) \otimes \xi \rightarrow \Lambda _{J}^{0,p+1}\left( \xi \right)
\otimes \xi $ is the extension by linearity of%
\begin{equation*}
\overline{\partial }\left( \alpha \otimes Z\right) =\overline{\partial }%
\alpha \otimes Z+\left( -1\right) ^{p}\alpha \wedge \overline{\partial }Z;
\end{equation*}%
ii) Let $\alpha \in \Lambda ^{1}\left( \xi \right) $,$\ \beta \in \Lambda
^{1}\left( \xi \right) \otimes \xi $ and $V\in \xi $. Then we denote $\alpha
^{0,1}\in \Lambda _{J}^{0,1}\left( \xi \right) $ and $\beta ^{0,1}\in
\Lambda _{J}^{0,1}\left( \xi \right) \otimes \xi $ the forms defined by%
\begin{equation*}
\alpha ^{0,1}\left( V\right) =\frac{1}{2}\left( \alpha \left( V\right)
+i\alpha \left( JV\right) \right) ,\ \ \beta ^{0,1}\left( V\right) =\frac{1}{%
2}\left( \beta \left( V\right) +J\beta \left( JV\right) \right) ;
\end{equation*}%
iii) Let $\alpha \in \Lambda _{J}^{0,1}\left( \xi \right) $, $V,W\in \xi $.
Then $\overline{\partial }a\in \Lambda _{J}^{0,2}\left( \xi \right) $ is
defined by $\overline{\partial }\alpha \left( V,W\right) =\frac{1}{4}d\alpha
^{\mathbb{C}}\left( V+iJV,W+iJW\right) $ where $d\alpha ^{\mathbb{C}}\in
\Lambda _{J}^{2}\xi \otimes \mathbb{C}$ is the complexification of $d\alpha $%
.
\end{notation}

\begin{lemma}
\label{dbar 1 form}Let $\omega \in \Lambda _{J}^{0,1}\left( \xi \right)
\otimes \xi $, $V,W\in \xi $. Then%
\begin{equation*}
\overline{\partial }\omega \left( V,W\right) =\left( \overline{\partial }%
\left( \omega \left( W\right) \right) \left( V\right) -\overline{\partial }%
\left( \omega \left( V\right) \right) \left( W\right) \right) -\frac{1}{2}%
\omega \left( \left[ V,W\right] -\left[ JV,JW\right] \right) .
\end{equation*}
\end{lemma}

\begin{proof}
It is enough to prove the Lemma for $\omega =\alpha ^{0,1}\otimes Z$, where $%
\alpha \in \Lambda _{J}^{1}\left( \xi \right) $ and $Z\in \xi $.

Then%
\begin{eqnarray}
\overline{\partial }\left( \alpha ^{0,1}\otimes Z\right) \left( V,W\right)
&=&\left( \overline{\partial }\alpha ^{0,1}\otimes Z-\alpha ^{0,1}\wedge 
\overline{\partial }Z\right) \left( V,W\right)  \notag \\
&=&\frac{1}{4}d\alpha ^{\mathbb{C}}\left( V+iJV,W+iJW\right) Z-\alpha
^{0,1}\left( V\right) \left( \overline{\partial }Z\right) \left( W\right)
+\alpha ^{0,1}\left( W\right) \left( \overline{\partial }Z\right) \left(
V\right)  \label{dbar 1 form A}
\end{eqnarray}%
But%
\begin{eqnarray*}
d\alpha ^{\mathbb{C}}\left( V+iJV,W+iJW\right) &=&\left( V+iJV\right) \left(
\alpha ^{\mathbb{C}}\left( W+iJW\right) \right) -\left( W+iJW\right) \left(
\alpha ^{\mathbb{C}}\left( V+iJV\right) \right) -\alpha ^{\mathbb{C}}\left[
V+iJV,W+iJW\right] \\
&=&\left( V+iJV\right) \left( \alpha \left( W\right) +i\alpha \left(
JW\right) \right) -\left( W+iJW\right) \left( \alpha \left( V\right)
+i\alpha \left( JV\right) \right) \\
&&-\alpha ^{\mathbb{C}}\left( \left[ V,W\right] -\left[ JV,JW\right] +i\left[
JV,W\right] +\left[ V,JW\right] \right) \\
&=&V\left( \alpha \left( W\right) \right) -\left( JV\right) \alpha \left(
JW\right) +i\left( \left( JV\right) \left( \alpha \left( W\right) \right)
+V\left( \alpha \left( JW\right) \right) \right) \\
&&-\left( W\left( \alpha \left( V\right) \right) -\left( JW\right) \alpha
\left( JV\right) +i\left( \left( JW\right) \left( \alpha \left( V\right)
\right) +W\left( \alpha \left( JV\right) \right) \right) \right) \\
&&-\alpha \left( \left[ V,W\right] -\left[ JV,JW\right] \right) +i\alpha
\left( \left[ JV,W\right] +\left[ V,JW\right] \right)
\end{eqnarray*}%
and so%
\begin{eqnarray}
d\alpha ^{\mathbb{C}}\left( V+iJV,W+iJW\right) Z &=&V\left( \alpha \left(
W\right) \right) -\left( JV\right) \alpha \left( JW\right) Z+\left(
JV\right) \left( \alpha \left( W\right) \right) +V\left( \alpha \left(
JW\right) \right) JZ  \notag \\
&&-\left( \left( W\left( \alpha \left( V\right) \right) -\left( JW\right)
\alpha \left( JV\right) \right) Z+\left( JW\right) \left( \alpha \left(
V\right) \right) +W\left( \alpha \left( JV\right) \right) JZ\right)
\label{dbar 1 form B} \\
&&-\left( \alpha \left( \left[ V,W\right] -\left[ JV,JW\right] \right)
\right) Z+\alpha \left( \left[ JV,W\right] +\left[ V,JW\right] \right) JZ. 
\notag
\end{eqnarray}

We have also%
\begin{eqnarray}
\alpha ^{0,1}\left( V\right) \left( \overline{\partial }Z\right) \left(
W\right) &=&\frac{1}{4}\left( \alpha \left( V\right) +i\alpha \left(
JV\right) \right) \left( \left[ W,Z\right] +J\left[ JW,Z\right] \right) 
\notag \\
&=&\frac{1}{4}\left( \alpha \left( V\right) \left[ W,Z\right] +\alpha \left(
V\right) J\left[ JW,Z\right] +\alpha \left( JV\right) J\left[ W,Z\right]
-\alpha \left( JV\right) \left[ JW,Z\right] \right) ,  \label{dbar 1 form C}
\end{eqnarray}%
\begin{eqnarray}
\alpha ^{0,1}\left( W\right) \left( \overline{\partial }Z\right) \left(
V\right) &=&\frac{1}{4}\left( \alpha \left( W\right) +i\alpha \left(
JW\right) \right) \left( \left[ V,Z\right] +J\left[ JV,Z\right] \right) 
\notag \\
&=&\frac{1}{4}\left( \alpha \left( W\right) \left[ V,Z\right] +\alpha \left(
W\right) J\left[ JV,Z\right] +\alpha \left( JW\right) J\left[ V,Z\right]
-\alpha \left( JW\right) \left[ JV,Z\right] \right)  \label{dbar 1 form D}
\end{eqnarray}%
so from (\ref{dbar 1 form A}), (\ref{dbar 1 form B}), (\ref{dbar 1 form C})
and (\ref{dbar 1 form D}) it follows that%
\begin{eqnarray}
\overline{\partial }\left( \alpha ^{0,1}\otimes Z\right) \left( V,W\right)
&=&\frac{1}{4}\left( V\left( \alpha \left( W\right) \right) -\left(
JV\right) \alpha \left( JW\right) Z+\left( JV\right) \left( \alpha \left(
W\right) \right) +V\left( \alpha \left( JW\right) \right) JZ\right)  \notag
\\
&&-\frac{1}{4}\left( \left( W\left( \alpha \left( V\right) \right) -\left(
JW\right) \alpha \left( JV\right) \right) Z+\left( JW\right) \left( \alpha
\left( V\right) \right) +W\left( \alpha \left( JV\right) \right) JZ\right) 
\notag \\
&&-\frac{1}{4}\left( \alpha \left( \left[ V,W\right] -\left[ JV,JW\right]
\right) \right) Z+\alpha \left( \left[ JV,W\right] +\left[ V,JW\right]
\right) JZ  \notag \\
&&-\frac{1}{4}\left( \alpha \left( V\right) +i\alpha \left( JV\right)
\right) \left( \left[ W,Z\right] +J\left[ JW,Z\right] \right)
\label{dbar 1 form E} \\
&&+\frac{1}{4}\left( \alpha \left( W\right) +i\alpha \left( JW\right)
\right) \left( \left[ V,Z\right] +J\left[ JV,Z\right] \right) .  \notag
\end{eqnarray}%
Since%
\begin{eqnarray}
\overline{\partial }\left( \left( \alpha ^{0,1}\otimes Z\right) \left(
W\right) \right) \left( V\right) &=&\overline{\partial }\left( \alpha
^{0,1}\left( W\right) Z\right) \left( V\right) =\frac{1}{2}\left( \left[
V,\alpha ^{0,1}\left( W\right) Z\right] +J\left[ JV,\alpha ^{0,1}\left(
W\right) Z\right] \right)  \notag \\
&=&\frac{1}{4}\left( \left[ V,\left( \alpha \left( W\right) +i\alpha \left(
JW\right) \right) Z\right] +J\left[ JV,\left( \alpha \left( W\right)
+i\alpha \left( JW\right) \right) Z\right] \right)  \notag \\
&=&\frac{1}{4}\left( \left[ V,\alpha \left( W\right) Z+\alpha \left(
JW\right) JZ\right] +J\left[ JV,\alpha \left( W\right) Z+\alpha \left(
JW\right) JZ\right] \right)  \notag \\
&=&\frac{1}{4}\left( \alpha \left( W\right) \left[ V,Z\right] +\alpha \left(
JW\right) \left[ V,JZ\right] +\alpha \left( W\right) J\left[ JV,Z\right]
+\alpha \left( JW\right) J\left[ JV,JZ\right] \right)  \notag \\
&&+\frac{1}{4}\left( V\left( \alpha \left( W\right) \right) \right)
Z+V\left( \alpha \left( JW\right) JZ+\left( JV\right) \left( \alpha \left(
W\right) \right) JZ-\left( JV\right) \left( \alpha \left( JW\right) \right)
Z\right)  \label{dbar 1 form F}
\end{eqnarray}%
and similarly 
\begin{eqnarray}
\overline{\partial }\left( \left( \alpha ^{0,1}\otimes Z\right) \left(
W\right) \right) \left( V\right) &=&\frac{1}{4}\left( \alpha \left( V\right) %
\left[ W,Z\right] +\alpha \left( JV\right) \left[ W,JZ\right] +\alpha \left(
V\right) J\left[ JW,Z\right] +\alpha \left( JV\right) J\left[ JW,JZ\right]
\right)  \notag \\
&&+\frac{1}{4}\left( W\left( \alpha \left( V\right) \right) \right)
Z+W\left( \alpha \left( JV\right) JZ+\left( JW\right) \left( \alpha \left(
V\right) \right) JZ-\left( JW\right) \left( \alpha \left( JV\right) \right)
Z\right)  \label{dbar 1 form G}
\end{eqnarray}%
from (\ref{dbar 1 form E}), (\ref{dbar 1 form F}) and (\ref{dbar 1 form G})
we obtain%
\begin{eqnarray*}
&&4\left( \overline{\partial }\left( \alpha ^{0,1}\otimes Z\right) \left(
V,W\right) -4\left( \overline{\partial }\left( \left( \alpha ^{0,1}\otimes
Z\right) \left( W\right) \right) \left( V\right) -\overline{\partial }\left(
\left( \alpha ^{0,1}\otimes Z\right) \left( V\right) \right) \left( W\right)
\right) \right) \\
&=&\underset{1}{V\left( \alpha \left( W\right) \right) }-\underset{2}{\left(
JV\right) \left( \alpha \left( JW\right) \right) }Z+\underset{3}{\left(
JV\right) \left( \alpha \left( W\right) \right) }+\underset{4}{V\left(
\alpha \left( JW\right) \right) JZ} \\
&&-\left( \left( \underset{5}{W\left( \alpha \left( V\right) \right) }-%
\underset{6}{\left( JW\right) \alpha \left( JV\right) }\right) Z+\underset{7}%
{\left( JW\right) \left( \alpha \left( V\right) \right) }+\underset{8}{%
W\left( \alpha \left( JV\right) \right) JZ}\right) \\
&&-\left( \alpha \left( \left[ V,W\right] -\left[ JV,JW\right] \right)
\right) Z+\alpha \left( \left[ JV,W\right] +\left[ V,JW\right] \right) JZ \\
&&-\left( \underset{9}{\alpha \left( V\right) \left[ W,Z\right] }+\underset{%
10}{\alpha \left( V\right) J\left[ JW,Z\right] }+\alpha \left( JV\right) J%
\left[ W,Z\right] -\alpha \left( JV\right) \left[ JW,Z\right] \right) \\
&&+\underset{13}{\alpha \left( W\right) \left[ V,Z\right] }+\underset{14}{%
\alpha \left( W\right) J\left[ JV,Z\right] }+\alpha \left( JW\right) J\left[
V,Z\right] -\alpha \left( JW\right) \left[ JV,Z\right] \\
&&-\left( \underset{13}{\alpha \left( W\right) \left[ V,Z\right] }+\alpha
\left( JW\right) \left[ V,JZ\right] +\underset{14}{\alpha \left( W\right) J%
\left[ JV,Z\right] }+\alpha \left( JW\right) J\left[ JV,JZ\right] \right) \\
&&-\underset{1}{\left( V\left( \alpha \left( W\right) \right) \right) Z}+%
\underset{4}{V\left( \alpha \left( JW\right) \right) JZ}+\underset{3}{\left(
JV\right) \left( \alpha \left( W\right) \right) JZ}-\underset{2}{\left(
JV\right) \left( \alpha \left( JW\right) \right) Z} \\
&&+\underset{9}{\alpha \left( V\right) \left[ W,Z\right] }+\alpha \left(
JV\right) \left[ W,JZ\right] +\underset{10}{\alpha \left( V\right) J\left[
JW,Z\right] }+\alpha \left( JV\right) J\left[ JW,JZ\right] \\
&&+\underset{5}{\left( W\left( \alpha \left( V\right) \right) \right) Z}+%
\underset{8}{W\left( \alpha \left( JV\right) \right) JZ}+\underset{7}{\left(
JW\right) \left( \alpha \left( V\right) \right) JZ}-\underset{6}{\left(
JW\right) \left( \alpha \left( JV\right) \right) Z}.
\end{eqnarray*}%
By reducing the terms having the same index, this last formula becomes%
\begin{eqnarray}
&&4\left( \overline{\partial }\left( \alpha ^{0,1}\otimes Z\right) \left(
V,W\right) -4\left( \overline{\partial }\left( \left( \alpha ^{0,1}\otimes
Z\right) \left( W\right) \right) \left( V\right) -\overline{\partial }\left(
\left( \alpha ^{0,1}\otimes Z\right) \left( V\right) \right) \left( W\right)
\right) \right)  \notag \\
&=&-\left( \alpha \left( \left[ V,W\right] -\left[ JV,JW\right] \right)
\right) Z+\alpha \left( \left[ JV,W\right] +\left[ V,JW\right] \right) JZ 
\notag \\
&&+\alpha \left( JV\right) \left( -J\left[ W,Z\right] +\left[ JW,Z\right] +%
\left[ W,JZ\right] +J\left[ JW,JZ\right] \right)  \notag \\
&&+\alpha \left( JW\right) \left( J\left[ V,Z\right] -\left[ JV,Z\right] -%
\left[ V,JZ\right] -J\left[ JV,JZ\right] \right)  \notag \\
&&  \label{dbar 1 form H}
\end{eqnarray}%
Since $N_{J}\left( V,Z\right) =N_{J}\left( W,Z\right) =0$, we have%
\begin{equation*}
J\left[ JW,JZ\right] -J\left[ W,Z\right] +\left[ JW,Z\right] +\left[ W,JZ%
\right] =0
\end{equation*}%
\begin{equation*}
J\left[ JV,JZ\right] -J\left[ V,Z\right] +\left[ V,JZ\right] +\left[ JV,Z%
\right] =0
\end{equation*}%
and from (\ref{dbar 1 form H}) it follows that.%
\begin{eqnarray}
&&4\left( \overline{\partial }\left( \alpha ^{0,1}\otimes Z\right) \left(
V,W\right) -4\left( \overline{\partial }\left( \left( \alpha ^{0,1}\otimes
Z\right) \left( W\right) \right) \left( V\right) -\overline{\partial }\left(
\left( \alpha ^{0,1}\otimes Z\right) \left( V\right) \right) \left( W\right)
\right) \right)  \notag \\
&=&-\left( \alpha \left( \left[ V,W\right] -\left[ JV,JW\right] \right)
\right) Z+\alpha \left( \left[ JV,W\right] +\left[ V,JW\right] \right) JZ. 
\notag \\
&&  \label{dbar 1 form I}
\end{eqnarray}%
But%
\begin{eqnarray*}
2\left( \alpha ^{0,1}\otimes Z\right) \left( \left[ V,W\right] -\left[ JV,JW%
\right] \right) &=&\left( \alpha \left( \left[ V,W\right] -\left[ JV,JW%
\right] \right) \right) +i\alpha \left( J\left( \left[ V,W\right] -\left[
JV,JW\right] \right) \right) Z \\
&=&\left( \left( \alpha \left( \left[ V,W\right] -\left[ JV,JW\right]
\right) Z+\alpha \left( J\left( \left[ V,W\right] -\left[ JV,JW\right]
\right) \right) \right) JZ\right)
\end{eqnarray*}%
and by (\ref{dbar 1 form I}) the Lemma follows.\nolinebreak
\end{proof}

\begin{definition}
\label{TY and HY}Let $Y$ a vector field on $L$. Define $T_{Y}=T_{Y,\gamma
,X}\in Hom_{\mathbb{C}}\left( \xi ,\xi \right) $ and $H_{Y}=H_{Y,J,\gamma
,X}\in \Lambda _{J}^{0,1}\left( \xi \right) \otimes \xi $ by%
\begin{equation}
T_{Y}\left( V\right) =\left[ V,Y\right] -\gamma \left( \left[ V,Y\right]
\right) X,  \label{TY}
\end{equation}%
\begin{equation}
H_{Y}\left( V\right) =\frac{1}{2}\left( T_{Y}\left( V\right) +JT_{Y}\left(
JV\right) \right) =\frac{1}{2}\left( \left[ V,Y\right] -\gamma \left( \left[
V,Y\right] \right) X+J\left( \left[ JV,Y\right] -\gamma \left( \left[ JV,Y%
\right] \right) X\right) \right) .  \label{HY1}
\end{equation}
\end{definition}

\begin{remark}
\label{Rem H}i)$H_{Y}$ is $C^{\infty }\left( L\right) $-linear. Indeed, let $%
a\in C^{\infty }\left( L\right) $. Then%
\begin{equation*}
T_{Y}\left( aV\right) =\left[ aV,Y\right] -\gamma \left( \left[ aV,Y\right]
\right) X=a\left[ V,Y\right] -Y\left( a\right) V-\gamma \left( a\left[ V,Y%
\right] -Y\left( a\right) V\right) X=aT\left( V\right) -Y\left( a\right) V
\end{equation*}%
and%
\begin{equation*}
H_{Y}\left( aV\right) =\frac{1}{2}\left( T_{Y}\left( aV\right) +JT_{Y}\left(
JaV\right) \right) =\frac{1}{2}\left( aT_{Y}\left( V\right) -Y\left(
a\right) V+aJT_{Y}\left( JV\right) -J\left( Y\left( a\right) JV\right)
\right) =aH_{Y}\left( V\right) .
\end{equation*}

ii) $H_{V}=\overline{\partial }V$ for $V\in \xi $.
\end{remark}

\begin{definition}
\label{TX and HX} In the particular case $Y=X$ of Definition \ref{TY and HY}%
, we will note $T=T_{X}=T_{\gamma ,X}$ and $H=H_{X}=H_{J,\gamma ,X}$. $%
T_{\gamma ,X}$ will be called the endomorphism associated to the DGLA
defining couple $\left( \gamma ,X\right) $ and $H_{J,\gamma ,X}$ will be
called the $\left( 0,1\right) $-form associated to the DGLA defining couple $%
\left( \gamma ,X\right) $.
\end{definition}

\begin{remark}
Let $V\in \xi $. We have%
\begin{equation}
H\left( V\right) =\frac{1}{2}\left( \left[ V,X\right] -\gamma \left( \left[
V,X\right] \right) X+J\left( \left[ JV,X\right] -\gamma \left( \left[ JV,X%
\right] \right) X\right) \right) .  \label{H0}
\end{equation}%
Since 
\begin{equation}
\iota _{X}d\gamma \left( V\right) =d\gamma \left( X,V\right) =X\left( \gamma
\left( V\right) \right) -V\left( \gamma \left( X\right) \right) -\gamma
\left( \left[ X,V\right] \right) =\gamma \left( \left[ V,X\right] \right)
\label{iXdgama(V)}
\end{equation}%
it follows that%
\begin{equation}
H\left( V\right) =\frac{1}{2}\left( \left[ V,X\right] +J\left[ JV,X\right]
\right) +\left( \iota _{X}d\gamma \right) ^{0,1}\left( V\right) X  \label{H}
\end{equation}
\end{remark}

\begin{lemma}
\label{dbar H}%
\begin{equation*}
\overline{\partial }H=\left( \iota _{X}d\gamma \right) ^{0,1}\wedge H
\end{equation*}%
where $H$ is the $\left( 0,1\right) $-form associated to the DGLA defining
couple $\left( \gamma ,X\right) $.
\end{lemma}

\begin{proof}
Let $V,W\in \mathbb{\xi }$. By Lemma \ref{dbar 1 form} we have 
\begin{eqnarray*}
2\overline{\partial }H\left( V,W\right) &=&2\left( \overline{\partial }%
\left( HW\right) \left( V\right) -\overline{\partial }\left( HV\right)
\left( W\right) \right) -H\left( \left[ V,W\right] -\left[ JV,JW\right]
\right) \\
&=&\left[ V,HW\right] +J\left[ JV,HW\right] -\left( \left[ W,HV\right] +J%
\left[ JW,HV\right] \right) -H\left( \left[ V,W\right] -\left[ JV,JW\right]
\right) .
\end{eqnarray*}%
By replacing $HV$, $HW$, $H\left( \left[ V,W\right] \right) $ and $H\left[
JV,JW\right] $ from the definition of $H$ we obtain that $2\overline{%
\partial }H\left( V,W\right) $ is the sum of the following $12$ terms: 
\begin{eqnarray}
4\overline{\partial }H\left( V,W\right) &=&\underset{\left( 1\right) }{\left[
V,\left[ W,X\right] -\gamma \left( \left[ W,X\right] \right) X\right] }+%
\underset{\left( 2\right) }{\left[ V,J\left( \left[ JW,X\right] -\gamma
\left( \left[ JW,X\right] \right) X\right) \right] }  \notag \\
&&+\underset{\left( 3\right) }{J\left( \left[ JV,\left[ W,X\right] -\gamma
\left( \left[ W,X\right] \right) X\right] \right) }+\underset{\left(
4\right) }{J\left[ JV,J\left( \left[ JW,X\right] -\gamma \left( \left[ JW,X%
\right] X\right) \right) \right] }  \notag \\
&&-\underset{\left( 5\right) }{\left[ W,\left[ V,X\right] -\gamma \left( %
\left[ V,X\right] \right) X\right] }-\underset{\left( 6\right) }{\left[
W,J\left( \left[ JV,X\right] -\gamma \left( \left[ JV,X\right] \right)
X\right) \right] }  \notag \\
&&-\underset{\left( 7\right) }{J\left( \left[ JW,\left[ V,X\right] -\gamma
\left( \left[ V,X\right] \right) X\right] \right) }-\underset{\left(
8\right) }{J\left[ JW,J\left( \left[ JV,X\right] -\gamma \left( \left[ JV,X%
\right] \right) X\right) \right] }  \notag \\
&&-\underset{\left( 9\right) }{\left[ \left[ V,W\right] ,X\right] -\gamma
\left( \left[ \left[ V,W\right] ,X\right] \right) X}-\underset{\left(
10\right) }{J\left( \left[ J\left[ V,W\right] ,X\right] -\gamma \left( \left[
J\left[ V,W\right] ,X\right] \right) X\right) }  \notag \\
&&+\underset{\left( 11\right) }{\left[ \left[ JV,JW\right] ,X\right] -\gamma
\left( \left[ \left[ JV,JW\right] ,X\right] \right) X}+\underset{\left(
12\right) }{J\left( \left[ J\left[ JV,JW\right] ,X\right] -\gamma \left( %
\left[ J\left[ JV,JW\right] ,X\right] \right) X\right) }  \notag \\
&&  \label{dbarH1}
\end{eqnarray}%
By using that $N\left( V,\left[ JW,X\right] -\gamma \left( \left[ JW,X\right]
X\right) \right) =0$, $N\left( W,J\left( \left[ JV,X\right] -\gamma \left( %
\left[ JV,X\right] X\right) \right) \right) =0$ and respectively $N\left(
V,W\right) =0$ we have%
\begin{eqnarray}
&&\left[ V,J\left( \left[ JW,X\right] -\gamma \left( \left[ JW,X\right]
\right) X\right) \right] +J\left[ JV,J\left( \left[ JW,X\right] -\gamma
\left( \left[ JW,X\right] \right) X\right) \right]  \label{(33)} \\
&=&J\left[ V,\left[ JW,X\right] -\gamma \left( \left[ JW,X\right] X\right) %
\right] -\left[ JV,\left[ JW,X\right] -\gamma \left( \left[ JW,X\right]
X\right) \right] ,  \notag
\end{eqnarray}%
\begin{eqnarray}
&&-J\left[ JW,J\left( \left[ JV,X\right] -\gamma \left( \left[ JV,X\right]
\right) X\right) \right] -\left[ W,J\left( \left[ JV,X\right] -\gamma \left( %
\left[ JV,X\right] \right) X\right) \right]  \label{44} \\
&=&\left[ JW,\left[ JV,X\right] -\gamma \left( \left[ JV,X\right] \right) X%
\right] -J\left[ W,\left[ JV,X\right] -\gamma \left( \left[ JV,X\right]
\right) X\right] ,  \notag
\end{eqnarray}%
and%
\begin{equation}
\left[ J\left[ JV,JW\right] ,X\right] -\left[ J\left[ V,W\right] ,X\right] =-%
\left[ \left[ JV,W\right] ,X\right] -\left[ \left[ V,JW\right] ,X\right] ,
\label{55}
\end{equation}%
From (\ref{55}) we have also%
\begin{eqnarray}
&&J\left[ J\left[ JV,JW\right] ,X\right] -\gamma \left( \left[ J\left[ JV,JW%
\right] ,X\right] X\right) -J\left( \left[ J\left[ V,W\right] ,X\right]
-\gamma \left( \left[ J\left[ V,W\right] ,X\right] \right) X\right)  \notag
\\
&=&-J\left( \left[ \left[ JV,W\right] ,X\right] -\gamma \left( \left[ \left[
JV,W\right] ,X\right] \right) X\right) -J\left( \left[ \left[ V,JW\right] ,X%
\right] -\gamma \left( \left[ \left[ V,JW\right] ,X\right] \right) X\right) .
\notag \\
&&  \label{56}
\end{eqnarray}

Replacing now in (\ref{dbarH1}) the terms $\left( 2\right) +\left( 4\right) $
by (\ref{(33)}), the terms $\left( 6\right) +\left( 8\right) $ by (\ref{44})
and the terms $\left( 10\right) +\left( 12\right) $ by (\ref{55}) and (\ref%
{56}) we deduce 
\begin{eqnarray}
4\overline{\partial }H\left( V,W\right) &=&\underset{(\alpha )}{\left[ V,%
\left[ W,X\right] -\gamma \left( \left[ W,X\right] \right) X\right] }  \notag
\\
&&+\underset{\left( \beta \right) }{J\left[ V,\left[ JW,X\right] -\gamma
\left( \left[ JW,X\right] X\right) \right] }  \notag \\
&&\underset{\left( \gamma \right) }{J\left[ JV,\left[ W,X\right] -\gamma
\left( \left[ W,X\right] \right) X\right] }  \notag \\
&&\underset{\left( \delta \right) }{-\left[ JV,\left[ JW,X\right] -\gamma
\left( \left[ JW,X\right] X\right) \right] }  \notag \\
&&\underset{\left( \alpha \right) }{-\left[ W,\left[ V,X\right] -\gamma
\left( \left[ V,X\right] \right) X\right] }  \notag \\
&&+\underset{\left( \delta \right) }{\left[ JW,\left[ JV,X\right] -\gamma
\left( \left[ JV,X\right] X\right) \right] }  \label{dbarH2} \\
&&-\underset{\left( \beta \right) }{J\left[ JW,\left[ V,X\right] -\gamma
\left( \left[ V,X\right] \right) X\right] }  \notag \\
&&-\underset{\left( \gamma \right) }{J\left[ W,\left[ JV,X\right] -\gamma
\left( \left[ JV,X\right] \right) X\right] }  \notag \\
&&\underset{\left( \alpha \right) }{-\left[ \left[ V,W\right] ,X\right]
+\gamma \left( \left[ \left[ V,W\right] ,X\right] \right) X}  \notag \\
&&\underset{\left( \delta \right) }{+\left[ \left[ JV,JW\right] ,X\right]
-\gamma \left( \left[ \left[ JV,JW\right] ,X\right] \right) X}  \notag \\
&&-\underset{\left( \gamma \right) }{J\left( \left[ \left[ JV,W\right] ,X%
\right] -\gamma \left( \left[ \left[ JV,W\right] ,X\right] \right) X\right) }
\notag \\
&&-\underset{\left( \beta \right) }{J\left( \left[ \left[ V,JW\right] ,X%
\right] -\gamma \left( \left[ \left[ V,JW\right] ,X\right] \right) X\right) }%
.  \notag
\end{eqnarray}%
By applying the Jacobi identities (\ref{Jalfa}), (\ref{Jbeta}), (\ref{Jgama}%
), (\ref{Jdelta}) below 
\begin{eqnarray}
&&\left[ V,\left[ W,X\right] -\gamma \left( \left[ W,X\right] \right) X%
\right] -\left[ W,\left[ V,X\right] -\gamma \left( \left[ V,X\right] \right)
X\right] -\left[ \left[ V,W\right] ,X\right] +\gamma \left( \left[ \left[ V,W%
\right] ,X\right] \right) X  \notag \\
&=&-\left[ V,\gamma \left( \left[ W,X\right] \right) X\right] +\left[
W,\gamma \left( \left[ V,X\right] \right) X\right] +\gamma \left( \left[ %
\left[ V,W\right] ,X\right] \right) X  \notag \\
&&  \label{Jalfa}
\end{eqnarray}%
\begin{eqnarray}
&&\left[ V,\left[ JW,X\right] -\gamma \left( \left[ JW,X\right] X\right) %
\right] -\left[ JW,\left[ V,X\right] -\gamma \left( \left[ V,X\right]
\right) X\right] -\left[ \left[ V,JW\right] ,X\right]  \notag \\
&=&-\left[ V,\gamma \left( \left[ JW,X\right] X\right) \right] +\left[
JW,\gamma \left( \left[ V,X\right] \right) X\right] -\gamma \left( \left[ %
\left[ V,JW\right] ,X\right] \right) X  \notag \\
&&  \label{Jbeta}
\end{eqnarray}%
\begin{eqnarray}
&&\left[ JV,\left[ W,X\right] -\gamma \left( \left[ W,X\right] \right) X%
\right] -\left[ W,\left[ JV,X\right] -\gamma \left( \left[ JV,X\right]
\right) X\right] -\left( \left[ \left[ JV,W\right] ,X\right] -\gamma \left( %
\left[ \left[ JV,W\right] ,X\right] \right) X\right)  \notag \\
&=&-\left[ JV,\gamma \left( \left[ W,X\right] \right) X\right] +\left[
W,\gamma \left( \left[ JV,X\right] \right) X\right] +\gamma \left( \left[ %
\left[ JV,W\right] ,X\right] \right) X  \notag \\
&&  \label{Jgama}
\end{eqnarray}%
\begin{eqnarray}
&&-\left[ JV,\left[ JW,X\right] -\gamma \left( \left[ JW,X\right] X\right) %
\right] +\left[ JW,\left[ JV,X\right] -\gamma \left( \left[ JV,X\right]
X\right) \right] +\left[ \left[ JV,JW\right] ,X\right] -\gamma \left( \left[ %
\left[ JV,JW\right] ,X\right] \right) X  \notag \\
&=&\left[ JV,\gamma \left( \left[ JW,X\right] X\right) \right] -\left[
JW,\gamma \left( \left[ JV,X\right] X\right) \right] -\gamma \left( \left[ %
\left[ JV,JW\right] ,X\right] \right) X  \notag \\
&&  \label{Jdelta}
\end{eqnarray}%
for the pairs of $3$ terms denoted $\left( \alpha \right) $, $\left( \beta
\right) $, $\left( \gamma \right) $ and $\left( \delta \right) $ in (\ref%
{dbarH2}) we obtain%
\begin{eqnarray*}
4\overline{\partial }H\left( V,W\right) &=&-\left[ V,\gamma \left( \left[ W,X%
\right] \right) X\right] +\left[ W,\gamma \left( \left[ V,X\right] \right) X%
\right] +\gamma \left( \left[ \left[ V,W\right] ,X\right] \right) X \\
&&+J\left( -\left[ V,\gamma \left( \left[ JW,X\right] X\right) \right] +%
\left[ JW,\gamma \left( \left[ V,X\right] \right) X\right] -\gamma \left( %
\left[ \left[ V,JW\right] ,X\right] \right) X\right) \\
&&+J\left( -\left[ JV,\gamma \left( \left[ W,X\right] \right) X\right] +%
\left[ W,\gamma \left( \left[ JV,X\right] \right) X\right] +\gamma \left( %
\left[ \left[ JV,W\right] ,X\right] \right) X\right) \\
&&+\left[ JV,\gamma \left( \left[ JW,X\right] X\right) \right] -\left[
JW,\gamma \left( \left[ JV,X\right] X\right) \right] -\gamma \left( \left[ %
\left[ JV,JW\right] ,X\right] \right) X.
\end{eqnarray*}%
So%
\begin{equation*}
4\overline{\partial }H\left( V,W\right) =A+B+C+D
\end{equation*}%
where%
\begin{eqnarray*}
A &=&-\gamma \left( \left[ W,X\right] \right) \left[ V,X\right] -\underset{%
\left( \varphi \right) }{V\left( \gamma \left( \left[ W,X\right] \right)
\right) X} \\
&&+\gamma \left( \left[ V,X\right] \right) \left[ W,X\right] +\underset{%
\left( \varphi \right) }{W\left( \gamma \left( \left[ V,X\right] \right)
\right) X}+\underset{\left( \varphi \right) }{\gamma \left( \left[ \left[ V,W%
\right] ,X\right] \right) }X
\end{eqnarray*}%
\begin{eqnarray*}
B &=&J\left( -\gamma \left( \left[ JW,X\right] \right) \left[ V,X\right] -%
\underset{\left( \eta \right) }{V\left( \gamma \left( \left[ JW,X\right]
\right) \right) X}\right) \\
&&+J\left( \gamma \left( \left[ V,X\right] \right) \left[ JW,X\right] +%
\underset{\left( \eta \right) }{JW\left( \gamma \left( \left[ V,X\right]
\right) \right) X}-\underset{\left( \eta \right) }{\gamma \left( \left[ %
\left[ V,JW\right] ,X\right] \right) X}\right)
\end{eqnarray*}%
\begin{eqnarray*}
C &=&J\left( -\gamma \left( \left[ W,X\right] \right) \left[ JV,X\right] -%
\underset{\left( \psi \right) }{\left( JV\right) \left( \gamma \left( \left[
W,X\right] \right) \right) X}\right) \\
&&+J\left( \gamma \left( \left[ JV,X\right] \right) \left[ W,X\right] +%
\underset{\left( \psi \right) }{W\left( \gamma \left( \left[ JV,X\right]
\right) \right) X}+\underset{\left( \psi \right) }{\gamma \left( \left[ %
\left[ JV,W\right] ,X\right] \right) X}\right)
\end{eqnarray*}%
\begin{eqnarray*}
D &=&\gamma \left( \left[ JW,X\right] \right) \left[ JV,X\right] +\underset{%
\left( \omega \right) }{\left( JV\right) \left( \gamma \left( \left[ JW,X%
\right] \right) \right) X} \\
&&-\gamma \left( \left[ JV,X\right] \right) \left[ JW,X\right] -\underset{%
\left( \omega \right) }{\left( JW\right) \left( \gamma \left( \left[ JV,X%
\right] \right) \right) X}-\underset{\left( \omega \right) }{\gamma \left( %
\left[ \left[ JV,JW\right] ,X\right] \right) X}.
\end{eqnarray*}%
By Lemma \ref{i(X) d(gama ) db closed} $\iota _{X}d\gamma $ is $d_{b}$%
-closed, therefore by (\ref{db}) 
\begin{equation*}
d_{b}\iota _{X}d\gamma =d\iota _{X}d\gamma -\gamma \wedge \iota _{X}d\left(
\iota _{X}d\gamma \right) =0.
\end{equation*}%
It follows that%
\begin{equation}
d\iota _{X}d\gamma \left( V,W\right) =0.  \label{diXgama}
\end{equation}%
By using (\ref{iXdgama(V)}) we have%
\begin{eqnarray*}
d\left( \iota _{X}d\gamma \right) \left( V,W\right) &=&V\left( \left( \iota
_{X}d\gamma \right) \left( W\right) \right) -W\left( \left( \iota
_{X}d\gamma \right) V\right) -\left( \iota _{X}d\gamma \right) \left[ V,W%
\right] \\
&=&V\left( d\gamma \left( X,W\right) \right) -W\left( d\gamma \left(
X,V\right) \right) -d\gamma \left( X,\left[ V,W\right] \right) \\
&=&V\left( \gamma \left[ W,X\right] \right) -W\left( \gamma \left[ V,X\right]
\right) -\gamma \left( \left[ \left[ V,W\right] ,X\right] \right) .
\end{eqnarray*}

So from (\ref{diXgama}) it follows%
\begin{equation}
-V\left( \gamma \left[ W,X\right] \right) +W\left( \gamma \left[ V,X\right]
\right) +\gamma \left( \left[ \left[ V,W\right] ,X\right] \right) =0
\label{fi}
\end{equation}%
and similarly%
\begin{equation}
-\left( JV\right) \left( \gamma \left[ W,X\right] \right) +W\left( \gamma %
\left[ JV,X\right] \right) +\gamma \left( \left[ \left[ JV,W\right] ,X\right]
\right) =0,  \label{csi}
\end{equation}%
\begin{equation}
-V\left( \gamma \left[ JW,X\right] \right) +\left( JW\right) \left( \gamma %
\left[ V,X\right] \right) +\gamma \left( \left[ \left[ V,JW\right] ,X\right]
\right) =0,  \label{eta}
\end{equation}%
\begin{equation}
-\left( JV\right) \left( \gamma \left[ JW,X\right] \right) +\left( JW\right)
\left( \gamma \left[ JV,X\right] \right) +\gamma \left( \left[ \left[ JV,JW%
\right] ,X\right] \right) =0,  \label{omega}
\end{equation}

By using (\ref{fi}), (\ref{csi}), (\ref{eta}) and (\ref{omega}) the pairs of
terms denoted by $\left( \varphi \right) $, $\left( \psi \right) $, $\left(
\eta \right) $ and $\left( \omega \right) $ reduce in $A,B,C,D$ respectively
and the expression of $4\overline{\partial }H\left( V,W\right) $ becomes%
\begin{eqnarray}
4\overline{\partial }H\left( V,W\right) &=&-\gamma \left( \left[ W,X\right]
\right) \left[ V,X\right] +\gamma \left( \left[ V,X\right] \right) \left[ W,X%
\right]  \notag \\
&&+J\left( -\gamma \left( \left[ JW,X\right] \right) \left[ V,X\right]
+\gamma \left( \left[ V,X\right] \right) \left[ JW,X\right] \right)  \notag
\\
&&+J\left( -\gamma \left( \left[ W,X\right] \right) \left[ JV,X\right]
+\gamma \left( \left[ JV,X\right] \right) \left[ W,X\right] \right)  \notag
\\
&&+\gamma \left( \left[ JW,X\right] \right) \left[ JV,X\right] -\gamma
\left( \left[ JV,X\right] \right) \left[ JW,X\right]  \label{dbarH3}
\end{eqnarray}%
We compute now 
\begin{eqnarray*}
&&4\left( \left( \iota _{X}d\gamma \right) ^{0,1}\wedge H\right) \left(
V,W\right) \\
&=&4\left( \left( \iota _{X}d\gamma \right) ^{0,1}\left( V\right) HW-\left(
\iota _{X}d\gamma \right) ^{0,1}\left( W\right) HV\right) \\
&=&2\left( \left( \iota _{X}d\gamma \left( V\right) +i\iota _{X}d\gamma
\left( JV\right) \right) HW-\left( \iota _{X}d\gamma \left( W\right) +i\iota
_{X}d\gamma \left( JW\right) \right) HV\right) \\
&=&2\left( \iota _{X}d\gamma \left( V\right) HW+\iota _{X}d\gamma \left(
JV\right) \right) JHW-\left( \iota _{X}d\gamma \left( W\right) HV+\iota
_{X}d\gamma \left( JW\right) \right) JHV.
\end{eqnarray*}%
By using (\ref{H0}) and (\ref{iXdgama(V)}) we obtain 
\begin{eqnarray*}
&&4\left( \left( \iota _{X}d\gamma \right) ^{0,1}\wedge H\right) \left(
V,W\right) \\
&=&\gamma \left( \left[ V,X\right] \right) \left( \left[ W,X\right] -\gamma
\left( \left[ W,X\right] \right) X+J\left( \left[ JW,X\right] -\gamma \left( %
\left[ JW,X\right] \right) X\right) \right) \\
&&+\gamma \left( \left[ JV,X\right] \right) J\left( \left[ W,X\right]
-\gamma \left( \left[ W,X\right] X\right) -\left( \left[ JW,X\right] -\gamma
\left( \left[ JW,X\right] \right) X\right) \right) \\
&&-\gamma \left( \left[ W,X\right] \right) \left( \left[ V,X\right] -\gamma
\left( \left[ V,X\right] \right) X+J\left( \left[ JV,X\right] -\gamma \left( %
\left[ JV,X\right] X\right) \right) \right) \\
&&-\gamma \left( \left[ JW,X\right] \right) \left( J\left( \left[ V,X\right]
-\gamma \left( \left[ V,X\right] \right) X\right) -\left( \left[ JV,X\right]
-\gamma \left[ JV,X\right] X\right) \right) \\
&=&\gamma \left( \left[ V,X\right] \right) \left[ W,X\right] -\underset{%
\left( 1\right) }{\gamma \left( \left[ V,X\right] \right) \gamma \left( %
\left[ W,X\right] \right) X} \\
&&-\gamma \left( \left[ JV,X\right] \right) \left[ JW,X\right] +\underset{%
\left( 2\right) }{\gamma \left( \left[ JV,X\right] \right) \gamma \left( %
\left[ JW,X\right] \right) X} \\
&&-\gamma \left( \left[ W,X\right] \right) \left[ V,X\right] +\underset{%
\left( 1\right) }{\gamma \left( \left[ W,X\right] \right) \gamma \left( %
\left[ V,X\right] \right) X} \\
&&+\gamma \left( \left[ JW,X\right] \right) \left[ JV,X\right] -\underset{%
\left( 2\right) }{\gamma \left( \left[ JW,X\right] \right) \gamma \left[ JV,X%
\right] X} \\
&&+J\left( \gamma \left( \left[ V,X\right] \right) \left[ JW,X\right] -%
\underset{\left( 3\right) }{\gamma \left( \left[ V,X\right] \right) \gamma
\left( \left[ JW,X\right] \right) X}\right) \\
&&+J\left( \gamma \left( \left[ JV,X\right] \right) \left[ W,X\right] -%
\underset{\left( 4\right) }{\gamma \left( \left[ JV,X\right] \right) \gamma
\left( \left[ W,X\right] X\right) }\right) \\
&&-J\left( \gamma \left( \left[ W,X\right] \right) \left[ JV,X\right] -%
\underset{\left( 4\right) }{\gamma \left( \left[ W,X\right] \right) \gamma
\left( \left[ JV,X\right] X\right) }\right) \\
&&-J\left( \gamma \left( \left[ JW,X\right] \right) \left[ V,X\right] -%
\underset{\left( 3\right) }{\gamma \left( \left[ JW,X\right] \right) \gamma
\left( \left[ V,X\right] \right) X}\right) .
\end{eqnarray*}

After reducing the pairs of terms $\left( 1\right) ,\left( 2\right) ,\left(
3\right) ,\left( 4\right) $, this expression coincides with (\ref{dbarH3})
and the Lemma is proved.
\end{proof}

\begin{proposition}
\label{change couple}Let $\left( \gamma ,X\right) $,$\left( \widehat{\gamma }%
,\widehat{X}\right) $ be $DGLA$-defining couples,$\ \widehat{\gamma }%
=e^{\lambda }\gamma $, $\widehat{X}\ =e^{-\lambda }X+U$, $\lambda \in
C^{\infty }\left( L\right) $, $U\in \xi $. Then: 
\begin{equation*}
H_{J,\widehat{\gamma },\widehat{X}}=e^{-\lambda }H_{J,\gamma ,X}+\overline{%
\partial }U-\left( \left( \iota _{X}d\gamma \right) ^{0,1}-\overline{%
\partial }\lambda \right) \otimes U.
\end{equation*}
\end{proposition}

\begin{proof}
Let $V\in \xi $. We have%
\begin{eqnarray*}
T_{\widehat{\gamma },\widehat{X}}\left( V\right) &=&\left[ V,e^{-\lambda }X+U%
\right] -e^{\lambda }\gamma \left( \left[ V,e^{-\lambda }X+U\right] \right)
\left( e^{-\lambda }X+U\right) \\
&=&e^{-\lambda }\left[ V,X\right] +V\left( e^{-\lambda }\right) X-\gamma
\left( e^{-\lambda }\left[ V,X\right] +V\left( e^{-\lambda }\right) X\right)
X \\
&&+\left[ V,U\right] -\gamma \left( \left[ V,U\right] \right) X-e^{\lambda
}\gamma \left( e^{-\lambda }\left[ V,X\right] +V\left( e^{-\lambda }\right)
X+\left[ V,U\right] \right) U.
\end{eqnarray*}%
Since $\left[ V,U\right] \in \xi $ and $\gamma \left( X\right) =1$ we obtain%
\begin{equation*}
T_{\widehat{\gamma },\widehat{X}}\left( V\right) =e^{-\lambda }T_{\gamma
,X}\left( V\right) +\left[ V,U\right] -\left( \gamma \left[ V,X\right]
+e^{\lambda }V\left( e^{-\lambda }\right) \right) U.
\end{equation*}%
It follows that%
\begin{eqnarray*}
H_{J,\widehat{\gamma },\widehat{X}}\left( V\right) &=&\frac{1}{2}\left( T_{%
\widehat{\gamma },\widehat{X}}\left( V\right) +JT_{\widehat{\gamma },%
\widehat{X}}\left( JV\right) \right) \\
&=&e^{-\lambda }H_{J,\gamma ,X}\left( V\right) +\frac{1}{2}\left( \left[ V,U%
\right] +J\left( \left[ JV,U\right] \right) \right) -\frac{1}{2}\left(
\gamma \left[ V,X\right] U+\gamma \left( \left[ JV,X\right] \right) JU\right)
\\
&&-\frac{1}{2}e^{\lambda }V\left( e^{-\lambda }\right) U-\frac{1}{2}%
e^{\lambda }JV\left( e^{-\lambda }\right) JU \\
&=&e^{-\lambda }H_{J,\gamma ,X}\left( V\right) +\overline{\partial }U\left(
V\right) -\frac{1}{2}\left( \gamma \left[ V,X\right] U+\gamma \left( \left[
JV,X\right] \right) JU\right) \\
&&-e^{\lambda }\left( \left( d\left( e^{-\lambda }\right) \left( V\right)
\right) U+\left( d\left( e^{-\lambda }\right) \left( JV\right) \right)
JU\right) .
\end{eqnarray*}%
By (\ref{iXdgama(V)}) we obtain%
\begin{eqnarray*}
H_{J,\widehat{\gamma },\widehat{X}}\left( V\right) &=&e^{-\lambda
}H_{J,\gamma ,X}\left( V\right) +\overline{\partial }U\left( V\right) -\frac{%
1}{2}\left( \iota _{X}d\gamma \left( V\right) U+\iota _{X}d\gamma \left(
JV\right) JU\right) \\
&&+\frac{1}{2}\left( \left( d\lambda \left( V\right) \right) U+\left(
d\lambda \left( JV\right) \right) JU\right) \\
&=&e^{-\lambda }H_{J,\gamma ,X}\left( V\right) +\overline{\partial }U\left(
V\right) -\left( \left( \left( \iota _{X}d\gamma -d\lambda \right)
^{0,1}\left( V\right) \right) U\right) \\
&=&\left( e^{-\lambda }H_{J,\gamma ,X}+\overline{\partial }U+\left( \left(
\iota _{X}d\gamma \right) ^{0,1}+\overline{\partial }\lambda \right) \otimes
U\right) \left( V\right)
\end{eqnarray*}%
and the Proposition is proved.
\end{proof}

\section{The $\mathfrak{\overline{\mathfrak{\beth }}}$-complex}

\begin{remark}
\label{i(X) d(gama )(0,1) dbar closed}$\left( \iota _{X}d\gamma \right)
^{0,1}$ is $\overline{\partial }_{J}$-closed. Indeed, by Lemma \ref{i(X)
d(gama ) db closed},%
\begin{equation*}
\overline{\partial }_{J}\left( \iota _{X}d\gamma \right) ^{0,1}=\left(
d_{b}\iota _{X}d\gamma \right) ^{0,2}=0.
\end{equation*}
\end{remark}

\begin{definition}
Let 
\begin{equation*}
\mathfrak{\overline{\mathfrak{\beth }}}=\mathfrak{\overline{\mathfrak{\beth }%
}}_{J,\gamma ,X}:\Lambda ^{0,\ast }\left( \xi \right) \otimes \xi
\rightarrow \Lambda ^{0,\ast }\left( \xi \right) \otimes \xi
\end{equation*}%
defined by%
\begin{equation*}
\mathfrak{\overline{\mathfrak{\beth }}}P=\overline{\partial }_{J}P-\left(
\iota _{X}d\gamma \right) ^{0,1}\wedge P,\ P\in \Lambda ^{0,p}\left( \xi
\right) \otimes \xi
\end{equation*}
\end{definition}

\begin{proposition}
\label{H fermee}a) $\mathfrak{\overline{\mathfrak{\beth }}}^{2}=0$;

b) $\mathfrak{\overline{\mathfrak{\beth }}}H=0$.
\end{proposition}

\begin{proof}
a) By Remark \ref{i(X) d(gama )(0,1) dbar closed} we have%
\begin{eqnarray*}
\mathfrak{\overline{\mathfrak{\beth }}}^{2}P &=&\overline{\partial }_{J}%
\mathfrak{\overline{\mathfrak{\beth }}}P-\left( \iota _{X}d\gamma \right)
^{0,1}\wedge \mathfrak{\overline{\mathfrak{\beth }}}P \\
&=&\overline{\partial }_{J}\left( \overline{\partial }_{J}P-\left( \iota
_{X}d\gamma \right) ^{0,1}\wedge P\right) -\left( \iota _{X}d\gamma \right)
^{0,1}\wedge \left( \overline{\partial }_{J}P-\left( \iota _{X}d\gamma
\right) ^{0,1}\wedge P\right) \\
&=&\left( \iota _{X}d\gamma \right) ^{0,1}\wedge \overline{\partial }%
_{J}P-\left( \iota _{X}d\gamma \right) ^{0,1}\wedge \overline{\partial }%
_{J}P=0
\end{eqnarray*}%
.

b) follows by Lemma \ref{dbar H}.
\end{proof}

\begin{proposition}
\label{Iso cohomologie}Let $\left( \gamma ,X\right) $,$\left( \widehat{%
\gamma },\widehat{X}\right) $ be $DGLA$-defining couples,$\ \widehat{\gamma }%
=e^{\lambda }\gamma $, $\widehat{X}\ =e^{-\lambda }X+U$, $\lambda \in
C^{\infty }\left( L\right) $, $U\in \xi $. Then:

\begin{equation*}
\mathfrak{\overline{\mathfrak{\beth }}}_{J,\widehat{\gamma },\widehat{X}%
}e^{-\lambda }P=e^{-\lambda }\mathfrak{\overline{\mathfrak{\beth }}}%
_{J,\gamma ,X}P,\ P\in \Lambda ^{0,\ast }\left( \xi \right) \otimes \xi .
\end{equation*}

In particular the map $\Phi :\alpha \mapsto e^{-\lambda }\alpha $ induces an
isomorphism $\widehat{\Phi }:H^{0,\ast }\left( \Lambda ^{0,\ast }\left( \xi
\right) \otimes \xi ,\mathfrak{\overline{\mathfrak{\beth }}}_{J,\gamma
,X}\right) \rightarrow H^{0,\ast }\left( \Lambda ^{0,\ast }\left( \xi
\right) \otimes \xi ,\mathfrak{\overline{\mathfrak{\beth }}}_{J,\widehat{%
\gamma },\widehat{X}}\right) $.
\end{proposition}

\begin{proof}
Let $V\in \xi $. Since $\xi =\ker \gamma $ and $V,U,\left[ V,U\right] \in
\xi $ and $\gamma \left( X\right) =1$ we have%
\begin{eqnarray*}
\iota _{\widehat{X}}d\widehat{\gamma }\left( V\right) &=&d\left( e^{\lambda
}\gamma \right) \left( e^{-\lambda }X+U,V\right) =e^{\lambda }d\gamma \left(
e^{-\lambda }X+U,V\right) +e^{\lambda }d\lambda \wedge \gamma \left(
e^{-\lambda }X+U,V\right) \\
&=&e^{\lambda }\left( \left( -V\left( \gamma \left( e^{-\lambda }X+U\right)
\right) -\gamma \left[ e^{-\lambda }X+U,V\right] \right) -d\lambda \left(
V\right) \gamma \left( e^{-\lambda }X+U\right) \right) \\
&=&e^{\lambda }\left( \left( -V\left( e^{-\lambda }\right) -e^{-\lambda
}\gamma \left[ X,V\right] \right) +V\left( e^{-\lambda }\right) -e^{-\lambda
}d\lambda \left( V\right) \right) \\
&=&-\gamma \left[ X,V\right] -d\lambda \left( V\right) =\iota _{X}d\gamma
\left( V\right) -d\lambda \left( V\right) .
\end{eqnarray*}%
So%
\begin{eqnarray*}
\mathfrak{\overline{\mathfrak{\beth }}}_{J,\widehat{\gamma },\widehat{X}}P
&=&\overline{\partial }_{J}P-\left( \iota _{\widehat{X}}d\widehat{\gamma }%
\right) ^{0,1}\wedge P=\overline{\partial }_{J}P-\left( \iota _{X}d\gamma
-d\lambda \right) ^{0,1}\wedge P \\
&=&\mathfrak{\overline{\mathfrak{\beth }}}_{J,\gamma ,X}P+\overline{\partial 
}\lambda \wedge P
\end{eqnarray*}%
and%
\begin{eqnarray*}
\mathfrak{\overline{\mathfrak{\beth }}}_{J,\widehat{\gamma },\widehat{X}%
}\left( e^{-\lambda }P\right) &=&\mathfrak{\overline{\mathfrak{\beth }}}%
_{J,\gamma ,X}\left( e^{-\lambda }P\right) +\overline{\partial }\lambda
\wedge e^{-\lambda }P=\overline{\partial }_{J}\left( e^{-\lambda }P\right)
-\left( \iota _{X}d\gamma \right) ^{0,1}\wedge e^{-\lambda }P+\overline{%
\partial }\lambda \wedge e^{-\lambda }P \\
&=&e^{-\lambda }\overline{\partial }_{J}P-e^{-\lambda }\overline{\partial }%
\lambda \wedge e^{-\lambda }P-e^{-\lambda }\left( \iota _{X}d\gamma \right)
^{0,1}\wedge P+e^{-\lambda }\overline{\partial }\lambda \wedge P=e^{-\lambda
}\mathfrak{\overline{\mathfrak{\beth }}}_{J,\gamma ,X}P.
\end{eqnarray*}
\end{proof}

\section{Exact Levi flat structures}

\begin{lemma}
\label{exact}Let $\left( \xi ,J\right) $ a Levi flat structure. We denote $%
\left[ H_{J,\gamma ,X}\right] \in H^{0,1}\left( \Lambda _{J}^{0,\ast }\left(
\xi \right) \otimes \xi ,\mathfrak{\overline{\mathfrak{\beth }}}_{J,\gamma
,X}\right) $ the cohomology class of the $\left( 0,1\right) $-form $%
H_{J,\gamma ,X}$ \ associated to a $DGLA$ defining couple $\left( \gamma
,X\right) $. The following are equivalent:

i) There exists a $DGLA$- definining couple $\left( \gamma ,X\right) $ such
that $H_{J,\gamma ,X}=0$.

ii) $\left[ H_{J,\gamma ,X}\right] =0$ for every $DGLA$ definining couple $%
\left( \gamma ,X\right) $.

iii) There exists a $DGLA$- definining couple $\left( \gamma ,X\right) $
such that $\left[ H_{J,\gamma ,X}\right] =0$.
\end{lemma}

\begin{proof}
$i)\implies ii)$. Let $\left( \gamma ,X\right) $ be a $DGLA$- definining
couple such that $H_{J,\gamma ,X}=0$. Let $\left( \widehat{\gamma },\widehat{%
X}\right) $ be a$\ DGLA$-defining couple,$\ \widehat{\gamma }=e^{\lambda
}\gamma $, $\widehat{X}\ =e^{-\lambda }X+U$, $\lambda \in C^{\infty }\left(
L\right) $, $U\in \xi $. By Proposition \ref{change couple} we have%
\begin{equation*}
H_{J,\widehat{\gamma },\widehat{X}}=\overline{\partial }U-\left( \left(
\iota _{X}d\gamma \right) ^{0,1}-\overline{\partial }\lambda \right) \otimes
U.
\end{equation*}%
But%
\begin{equation*}
\mathfrak{\overline{\mathfrak{\beth }}}_{J,\gamma ,X}\left( e^{\lambda
}U\right) =e^{\lambda }\left( \overline{\partial }_{J}U-\left( \left( \iota
_{X}d\gamma \right) ^{0,1}-\overline{\partial }\lambda \right) \otimes
U\right) =e^{\lambda }H_{J,\widehat{\gamma },\widehat{X}}
\end{equation*}%
and by Proposition \ref{Iso cohomologie} it follows that $H_{J,\widehat{%
\gamma },\widehat{X}}=\mathfrak{\overline{\mathfrak{\beth }}}_{J,\widehat{%
\gamma },\widehat{X}}U$.

$ii)\implies iii)$ is obvious.

$iii)\implies i)$. Let $\left( \gamma ,X\right) $ be a $DGLA$- definining
couple and $U\in \xi $ such that $H_{J,\gamma ,X}=\mathfrak{\overline{%
\mathfrak{\beth }}}_{J,\gamma ,X}U$. Let $\left( \widehat{\gamma },\widehat{X%
}\right) =\left( \left( \gamma ,X-U\right) \right) $. By Proposition \ref%
{change couple} 
\begin{equation*}
H_{J,\widehat{\gamma },\widehat{X}}=\mathfrak{\overline{\mathfrak{\beth }}}%
_{J,\gamma ,X}U-\left( \overline{\partial }U-\left( \iota _{X}d\gamma
\right) ^{0,1}\otimes U\right) =0.
\end{equation*}
\end{proof}

\begin{definition}
Let $\left( \xi ,J\right) $ be a Levi flat structure We say that $\left( \xi
,J\right) $ is exact if it verifies one of the equivalent conditions of
Lemma \ref{exact}.
\end{definition}

\begin{example}
1) Let $\left( M,J\right) $ be a complex manifold and $L=M\times S^{1}$.
Consider the Levi flat structure $\xi =\left( T\left( M\right) ,J\right) $ \
and let $\left( \gamma ,X\right) $ a $DGLA$ defining couple, where $\gamma
=d\theta $ and $X=\frac{\partial }{\partial \theta }$ where $\theta $ is a
coordinate on $S^{1}$. Then $H=0$, so the Levi flat structure is exact.

2) Let $\left( M,J\right) $ be a compact complex manifold and $L=\left\{
z\in M:\ r\left( z\right) =0\right\} $ a real analytic Levi flat
hypersurface in $M$ where $r$ a is real anlytic function and $dr\neq 0$ on $%
L $. Then the Levi flat structure $\left( T\mathbb{C}\left( L\right)
,J\right) $ is exact.

Indeed,let $g$ be a fixed Hermitian metric on $M$ and $Z=grad_{g}r/\left%
\Vert grad_{g}r\right\Vert _{g}^{2}$. Then $\left( \gamma ,X\right) =\left(
d_{J}^{c}r,JZ\right) $ is a DGLA defining couple for the Levi foliation \cite%
{Bartolomeis10}.

Let $p\in L$. There exists holomorphic coordinates $z=\left( z_{1},\cdot
\cdot \cdot ,z_{n}\right) $, $z_{j}=x_{2j-1}+ix_{2j}$, $j=1,\cdot \cdot
\cdot ,n$, in a neighborhood $U$ of $p$ such that $L=\left\{ z\in
U:x_{2n-1}=0\ \right\} $, so $r=-e^{\lambda }x_{2n-1}$ with $\lambda $ a
smooth function on $U$. Set $grad_{g}r=\sum_{i=1}^{2n}a_{i}\frac{\partial }{%
\partial x_{i}}$. We have $g\left( grad_{g}r,Y\right) =dr\left( Y\right) $
for every vector field $Y,$ so for every $\left( \alpha _{1},\cdot \cdot
\cdot ,\alpha _{2n}\right) \in \mathbb{R}^{2n}$ we have 
\begin{equation*}
g\left( \sum_{i=1}^{2n}a_{i}\frac{\partial }{\partial x_{i}}%
,\sum_{i=1}^{2n}\alpha _{i}\frac{\partial }{\partial x_{i}}\right)
=\sum_{i=1}^{2n}\alpha _{i}\frac{\partial r}{\partial x_{i}}%
=-\sum_{i=1}^{2n}\alpha _{i}e^{\lambda }\frac{\partial \lambda }{\partial
x_{i}}x_{2n-1}-e^{\lambda }\alpha _{2n-1}
\end{equation*}%
on $U$.\ In particular on $U\cap L$ we have $g\left( \sum_{i=1}^{2n}a_{i}%
\frac{\partial }{\partial x_{i}},\sum_{i=1}^{2n-1}\alpha _{i}\frac{\partial 
}{\partial x_{i}}\right) =0$ and it follows that $\alpha _{1}=\cdot \cdot
\cdot =\alpha _{2n-2}=\alpha _{2n}=0$ on $U\cap L$. So 
\begin{equation*}
X_{\left\vert U\cap L\right. }=J\frac{grad_{g}r}{\left\Vert
grad_{g}r\right\Vert _{g}^{2}}_{\left\vert U\cap L\right. }=e^{-\lambda }%
\frac{\partial }{\partial x_{2n}}.
\end{equation*}

Set now $z_{j}=x_{j}+iy_{j}$, $j=1,\cdot \cdot \cdot ,n$, and let $%
V=\sum_{j=1}^{n-1}a_{j}\frac{\partial }{\partial x_{j}}+\sum_{j=1}^{n-1}b_{j}%
\frac{\partial }{\partial y_{j}}\in \ker \gamma $. Then on $U\cap L$ we have 
\begin{eqnarray*}
\left[ V,X\right] &=&\left[ \sum_{j=1}^{n-1}a_{j}\frac{\partial }{\partial
x_{j}}+\sum_{j=1}^{n-1}b_{j}\frac{\partial }{\partial y_{j}},e^{-\lambda }%
\frac{\partial }{\partial y_{n}}\right] \\
&=&-e^{-\lambda }\left( \sum_{j=1}^{n-1}\frac{\partial a_{j}}{\partial y_{n}}%
\frac{\partial }{\partial x_{j}}+\sum_{j=1}^{n-1}\frac{\partial b_{j}}{%
\partial y_{n}}\frac{\partial }{\partial y_{j}}\right) +\left(
\sum_{j=1}^{n-1}a_{j}\frac{\partial }{\partial x_{j}}+\sum_{j=1}^{n-1}b_{j}%
\frac{\partial }{\partial y_{j}}\right) \left( e^{-\lambda }\right) \frac{%
\partial }{\partial y_{n}},
\end{eqnarray*}%
\begin{eqnarray*}
\left[ JV,X\right] &=&\left[ -\sum_{j=1}^{n-1}b_{j}\frac{\partial }{\partial
x_{j}}+\sum_{j=1}^{n-1}a_{j}\frac{\partial }{\partial y_{j}},e^{-\lambda }%
\frac{\partial }{\partial y_{n}}\right] \\
&=&=-e^{-\lambda }\left( -\sum_{j=1}^{n-1}\frac{\partial b_{j}}{\partial
y_{n}}\frac{\partial }{\partial x_{j}}+\sum_{j=1}^{n-1}\frac{\partial a_{j}}{%
\partial y_{n}}\frac{\partial }{\partial y_{j}}\right) +\left(
-\sum_{j=1}^{n-1}b_{j}\frac{\partial }{\partial x_{j}}+\sum_{j=1}^{n-1}a_{j}%
\frac{\partial }{\partial y_{j}}\right) \left( e^{-\lambda }\right) \frac{%
\partial }{\partial y_{n}},
\end{eqnarray*}%
so 
\begin{equation*}
\left[ V,X\right] -\gamma \left( \left[ V,X\right] \right) X=-e^{-\lambda
}\left( \sum_{j=1}^{n-1}\frac{\partial a_{j}}{\partial y_{n}}\frac{\partial 
}{\partial x_{j}}+\sum_{j=1}^{n-1}\frac{\partial b_{j}}{\partial y_{n}}\frac{%
\partial }{\partial y_{j}}\right)
\end{equation*}%
and%
\begin{equation*}
J\left( \left[ JV,X\right] -\gamma \left( \left[ JV,X\right] \right)
X\right) =-e^{-\lambda }\left( \sum_{j=1}^{n-1}-\frac{\partial a_{j}}{%
\partial y_{n}}\frac{\partial }{\partial x_{j}}-\sum_{j=1}^{n-1}\frac{%
\partial b_{j}}{\partial y_{n}}\frac{\partial }{\partial y_{j}}\right) .
\end{equation*}%
It follows that $H=0$.
\end{example}

\section{Deformation theory of Levi flat structures}

In this paragraph we a consider a fixed DGLA defining couple $\left( \gamma
,X\right) $ and $T=T_{\gamma ,X}$, $H=H_{J,\gamma ,X}$ will design the
endomorphism and respectively the $\left( 0,1\right) $-form associated to
the DGLA\ defining couple $\left( \gamma ,X\right) $ (Definition \ref{TX and
HX}).

Let $\alpha \in \mathcal{Z}^{1}\left( L\right) $ satisfying the
Maurer-Cartan equation. We start with a formula which describes the
deformation of the Lie bracket:

\begin{lemma}
\label{Rel [ ] and [ ]alfa} 
\begin{equation*}
\left[ \cdot ,\cdot \right] _{\alpha }=\left[ \cdot ,\cdot \right] +\alpha
\wedge T.
\end{equation*}
\end{lemma}

\begin{proof}
We have 
\begin{eqnarray}
\left( \alpha \wedge T\right) \left( V,W\right) &=&\alpha \left( V\right)
T\left( W\right) -\alpha \left( W\right) T\left( V\right)  \notag \\
&=&\alpha \left( V\right) \left( \left[ W,X\right] -\gamma \left( \left[ W,X%
\right] \right) X\right)  \notag \\
&&-\alpha \left( W\right) \left( \left[ V,X\right] -\gamma \left( \left[ V,X%
\right] \right) X\right)  \label{alfa lamda TX0} \\
&=&\alpha \left( V\right) \left[ W,X\right] -\alpha \left( W\right) \left[
V,X\right]  \notag \\
&&-\alpha \left( V\right) \gamma \left( \left[ W,X\right] \right) X+\alpha
\left( W\right) \gamma \left( \left[ V,X\right] \right) X.  \notag
\end{eqnarray}%
Since $\gamma \left( V\right) =\gamma \left( W\right) =0$ and $\gamma \left(
X\right) =1$ it follows that $\gamma \left( \left[ W,X\right] \right)
=d\gamma \left( X,W\right) $ and $\gamma \left( \left[ V,X\right] \right)
=d\gamma \left( X,V\right) $.

Replacing in (\ref{alfa lamda TX0}) we obtain%
\begin{equation}
\left( \alpha \wedge T\right) \left( V,W\right) =\alpha \left( V\right) 
\left[ W,X\right] -\alpha \left( W\right) \left[ V,X\right] -\left( \alpha
\left( V\right) d\gamma \left( X,W\right) -\alpha \left( W\right) d\gamma
\left( X,V\right) \right) X.  \label{alfa lamda TX}
\end{equation}

But $\alpha \in \mathcal{Z}^{1}\left( L\right) $ satisfies the Maurer-Cartan
equation, so%
\begin{equation*}
\delta \alpha +\frac{1}{2}\left\{ \alpha ,\alpha \right\} =d\alpha +\iota
_{X}\left( d\alpha \right) \wedge \alpha +\iota _{X}d\gamma \wedge \alpha
-\gamma \wedge \iota _{X}d\alpha =0
\end{equation*}%
and it follows that%
\begin{equation*}
d\alpha \left( V,W\right) =\left( -\iota _{X}\left( d\alpha \right) \wedge
\alpha -\iota _{X}d\gamma \wedge \alpha +\gamma \wedge \iota _{X}d\alpha
\right) \left( V,W\right) .
\end{equation*}%
Taking in account that $\gamma \left( V\right) =\gamma \left( W\right) =0,\
\alpha \left( X\right) =0$ we obtain%
\begin{eqnarray*}
d\alpha \left( V,W\right) &=&-d\alpha \left( X,V\right) \alpha \left(
W\right) +d\alpha \left( X,W\right) \alpha \left( V\right) \\
&&-d\gamma \left( X,V\right) \alpha \left( W\right) +d\gamma \left(
X,W\right) \alpha \left( V\right)
\end{eqnarray*}

so%
\begin{equation*}
\alpha \left( V\right) d\gamma \left( X,W\right) -\alpha \left( W\right)
d\gamma \left( X,V\right) =d\alpha \left( V,W\right) -\alpha \left( V\right)
d\alpha \left( X,W\right) +\alpha \left( W\right) d\alpha \left( X,V\right)
\end{equation*}%
and (\ref{alfa lamda TX}) becomes%
\begin{equation}
\left( \alpha \wedge T\right) \left( V,W\right) =\alpha \left( V\right) 
\left[ W,X\right] -\alpha \left( W\right) \left[ V,X\right] -\left( d\alpha
\left( V,W\right) -\alpha \left( V\right) d\alpha \left( X,W\right) +\alpha
\left( W\right) d\alpha \left( X,V\right) \right) X  \label{alfa lamda TX1}
\end{equation}

By (\ref{[V,W]alfa}) we have%
\begin{eqnarray}
\left[ V,W\right] _{\alpha }-\left[ V,W\right] &=&\alpha \left( V\right) %
\left[ W,X\right] -\alpha \left( W\right) \left[ V,X\right]  \notag \\
&&+\left( W\left( \alpha \left( V\right) \right) -V\left( \alpha \left(
W\right) \right) \right) X  \notag \\
&&+\alpha \left( V\right) X\left( \alpha \left( W\right) \right) X-\alpha
\left( W\right) \left( X\left( \alpha \left( V\right) \right) \right) X
\label{[V,W]alfa4} \\
&&-\alpha \left( V\right) \alpha \left( \left[ X,W\right] \right) X-\alpha
\left( W\right) \alpha \left( \left[ V,X\right] \right) X+\alpha \left( %
\left[ V,W\right] \right) X.  \notag
\end{eqnarray}%
Since $\alpha \left( X\right) =0$ we have 
\begin{equation*}
\alpha \left( V\right) X\left( \alpha \left( W\right) \right) -\alpha \left(
V\right) \alpha \left( \left[ X,W\right] \right) =\alpha \left( V\right)
d\alpha \left( X,W\right) ,
\end{equation*}%
\begin{equation*}
-\alpha \left( W\right) X\left( \alpha \left( V\right) \right) +\alpha
\left( W\right) \alpha \left( \left[ X,V\right] \right) =-\alpha \left(
W\right) d\alpha \left( X,V\right)
\end{equation*}%
and%
\begin{equation*}
W\left( \alpha \left( V\right) \right) -V\left( \alpha \left( W\right)
\right) -\alpha \left( \left[ W,V\right] \right) =-d\alpha \left( V,W\right)
,
\end{equation*}%
so (\ref{[V,W]alfa4}) becomes%
\begin{eqnarray}
\left[ V,W\right] _{\alpha }-\left[ V,W\right] &=&\alpha \left( V\right) %
\left[ W,X\right] -\alpha \left( W\right) \left[ V,X\right]  \notag \\
&&\left( \alpha \left( V\right) d\alpha \left( X,W\right) -\alpha \left(
W\right) d\alpha \left( X,V\right) -d\alpha \left( V,W\right) \right) X.
\label{[V,W]alfa3}
\end{eqnarray}%
The Lemma follows now from (\ref{alfa lamda TX1}) and (\ref{[V,W]alfa3}).
\end{proof}

\begin{corollary}
\label{Rel N alfa H}Let $N_{J}^{\alpha }$ be the Nijenhuis tensor for $%
\left( \mathcal{H}\left( \xi \right) ,\left[ \cdot ,\cdot \right] _{\alpha
},J\right) $. Then%
\begin{equation*}
N_{J}^{\alpha }=-4\alpha _{J}^{0,1}\wedge H.
\end{equation*}
\end{corollary}

\begin{proof}
Let $V,W\in \xi $. By Lemma \ref{Rel [ ] and [ ]alfa} 
\begin{equation*}
\left[ JV,JW\right] _{\alpha }=\left[ JV,JW\right] +\alpha \left( JV\right)
T\left( JW\right) -\alpha \left( JW\right) T\left( JV\right)
\end{equation*}%
\begin{equation*}
\left[ V,W\right] _{\alpha }=\left[ V,W\right] +\alpha \left( V\right)
T\left( W\right) -\alpha \left( W\right) T\left( V\right)
\end{equation*}%
\begin{equation*}
J\left[ JV,W\right] _{\alpha }=J\left[ JV,W\right] +J\left( \alpha \left(
JV\right) T\left( W\right) -\alpha \left( W\right) T\left( JV\right) \right)
\end{equation*}%
\begin{equation*}
J\left[ V,JW\right] _{\alpha }=J\left[ V,JW\right] +J\left( \alpha \left(
V\right) T\left( JW\right) -\alpha \left( JW\right) T\left( V\right) \right)
\end{equation*}%
so%
\begin{eqnarray*}
N_{J}^{\alpha }\left( V,W\right) &=&N_{J}\left( V,W\right) -\alpha \left(
V\right) \left( T\left( W\right) +JT\left( JW\right) \right) +\alpha \left(
W\right) \left( T\left( V\right) +JT\left( JV\right) \right) \\
&&-J\alpha \left( JV\right) \left( T\left( W\right) +JT\left( JW\right)
\right) +J\alpha \left( JW\right) \left( T\left( V\right) +JT\left(
JV\right) \right) \\
&=&N_{J}\left( V,W\right) -2\left( \alpha \left( V\right) HW-\alpha \left(
W\right) HV+\alpha \left( JV\right) JHW-\alpha \left( JW\right) JHV\right) .
\end{eqnarray*}%
Since $N_{J}=0$ and%
\begin{eqnarray*}
\left( \alpha ^{0,1}\wedge H\right) \left( V,W\right) &=&\frac{1}{2}\left(
\alpha \left( V\right) +i\alpha \left( JV\right) \right) HW-\frac{1}{2}%
\left( \alpha \left( W\right) +i\alpha \left( JW\right) \right) HV \\
&=&\frac{1}{2}\left( \alpha \left( V\right) HW+\alpha \left( JV\right)
JHW-\alpha \left( W\right) HV-\alpha \left( JW\right) \right) JHV,
\end{eqnarray*}
the Corollary follows.
\end{proof}

\begin{lemma}
\label{Rel N Jtilda alfa N J alfa}Suppose that $\alpha $ is close to $0$ and
let $J_{\alpha }$ be a complex structure on $\xi _{\alpha }$. Set $%
\widetilde{J_{\alpha }}=\omega _{\alpha }^{-1}J_{\alpha }\omega _{\alpha }$
and let $S_{\alpha }~$\ be the unique form in $\Lambda _{J}^{0,1}\left( \xi
\right) \otimes \xi $ such that the complex structure $\widetilde{J_{\alpha }%
}$ on $\xi $ is given by $\widetilde{J_{\alpha }}=\left( I+S_{\alpha
}\right) J\left( I+S_{\alpha }\right) ^{-1}$. Then $N_{J_{\alpha }}=0$ if
and only if $N_{\widetilde{J_{\alpha }}}^{\alpha }=0$.
\end{lemma}

\begin{proof}
For every $V,W\in \xi _{\alpha }$ we have 
\begin{equation*}
N_{\widetilde{J_{\alpha }}}^{\alpha }\left( \omega _{\alpha }^{-1}V,\omega
_{\alpha }^{-1}W\right) =N_{J_{\alpha }}\left( V,W\right)
\end{equation*}%
and the Lemma follows.
\end{proof}

From Lemma \ref{Ker(gama+alfa) integrable}, Lemma \ref{Rel N Jtilda alfa N J
alfa} Corollary \ref{N Jtilda=0} and Lemma \ref{Rel N Jtilda alfa N J alfa}
we obtain

\begin{corollary}
\label{Eq def Levi flat}Let $J_{\alpha }$ a complex structure on $\xi
_{\alpha }$. We denote $S_{\alpha }$ the unique form in $\Lambda
_{J}^{0,1}\left( \xi \right) \otimes \xi $ such that $\omega _{\alpha
}^{-1}J_{\alpha }\omega _{\alpha }=\left( I+S_{\alpha }\right) J\left(
I+S_{\alpha }\right) ^{-1}$.The following are equivalent:

1) $\left( \xi _{\alpha },J_{\alpha }\right) $ is a Levi flat structure.

2) 
\begin{equation}
\delta \alpha +\frac{1}{2}\left\{ \alpha ,\alpha \right\} =0  \label{Eq alfa}
\end{equation}%
and%
\begin{equation}
\overline{\partial }_{J}^{\alpha }S_{\alpha }+\frac{1}{2}\left[ \left[
S_{\alpha },S_{\alpha }\right] \right] _{\alpha }=\frac{1}{4}N_{J}^{\alpha
}=-\alpha _{J}^{0,1}\wedge H  \label{Eq S}
\end{equation}%
where $\overline{\partial }_{J}^{\alpha }=\overline{\partial }_{J,\left[
\cdot ,\cdot \right] _{\alpha }}$ and%
\begin{equation*}
\left[ \left[ S_{\alpha },S_{\alpha }\right] \right] _{\alpha }=\left[
S_{\alpha },S_{\alpha }\right] _{\alpha }-\frac{1}{4}S_{\alpha }\left(
N_{J}^{\alpha }-N_{J}^{\alpha }\left( S_{\alpha },S_{\alpha }\right) \right)
.
\end{equation*}
\end{corollary}

\section{Gauge equivalence}

We recall that the group action $\chi $ of $\mathcal{G=}Diff\left( L\right) $
on $\mathcal{Z}^{1}\left( L\right) $ is defined in (\ref{chi(fi)}) such that
for $\alpha \in \mathcal{Z}^{1}\left( L\right) $ and $\Phi \in \mathcal{G}$
the distribution $\xi _{\alpha }$ is integrable if and only if the
distribution $\xi _{\chi \left( \Phi \right) \left( \alpha \right) }=\Phi
_{\ast }\xi _{\alpha }$ is integrable (Remark \ref{group action}). If $%
J_{\alpha }$ is a complex structure on $\xi _{\alpha }$, by Lemma \ref{rel
complex structures} we define a complex structure 
\begin{equation*}
J_{\chi \left( \Phi \right) \left( \alpha \right) }=\omega _{\chi \left(
\Phi \right) \left( \alpha \right) }\left( I+S_{\chi \left( \Phi \right)
\left( \alpha \right) }\right) J\left( I+S_{\chi \left( \Phi \right) \left(
\alpha \right) }\right) ^{-1}\omega _{\chi \left( \Phi \right) \left( \alpha
\right) }^{-1}
\end{equation*}%
on $\xi _{\chi \left( \Phi \right) \left( \alpha \right) }$ where%
\begin{equation}
S_{\chi \left( \Phi \right) \left( \alpha \right) }=\left( J-\omega _{\chi
\left( \Phi \right) \left( \alpha \right) }^{-1}\Phi _{\ast }J_{\alpha }\Phi
_{\ast }^{-1}\omega _{\chi \left( \Phi \right) \left( \alpha \right)
}\right) \left( J+\omega _{\chi \left( \Phi \right) \left( \alpha \right)
}^{-1}\Phi _{\ast }J_{\alpha }\Phi _{\ast }^{-1}\omega _{\chi \left( \Phi
\right) \left( \alpha \right) }\right) ^{-1}.  \label{S chi Fi alfa}
\end{equation}%
With these definitions we have the following:

\begin{lemma}
\label{Gauge equivalence}\bigskip Let $\Phi \in \mathcal{G}$. Then $\left(
\xi _{\alpha },J_{\alpha }\right) $ is a Levi flat structure if and only if $%
\left( \xi _{\chi \left( \Phi \right) \left( \alpha \right) },J_{\chi \left(
\Phi \right) \left( \alpha \right) }\right) $ is a Levi flat structure.
\end{lemma}

\begin{proof}
By Lemma \ref{Rel N Jtilda alfa N J alfa} $N_{J_{\alpha }}=0$ if and only if 
$N_{\widetilde{J_{\alpha }}}^{\alpha }=0$ and \ $N_{J_{\chi \left( \Phi
\right) \left( \alpha \right) }}=0$ if and only if $N_{\widetilde{J\chi
\left( \Phi \right) \left( \alpha \right) }}^{\alpha }=0$, where $\widetilde{%
J_{\alpha }}=\omega _{\alpha }^{-1}J_{\alpha }\omega _{\alpha }$ and $\ 
\widetilde{J\chi \left( \Phi \right) \left( \alpha \right) }=\omega _{\chi
\left( \Phi \right) \left( \alpha \right) }^{-1}\Phi _{\ast }^{-1}J_{\alpha
}\Phi _{\ast }\omega _{\chi \left( \Phi \right) \left( \alpha \right) }$.
But $\ \omega _{\alpha }\left( \xi \right) =\xi _{\alpha }=\Phi _{\ast
}\left( \xi _{\chi \left( \Phi \right) \left( \alpha \right) }\right) $ and
the Lemma follows.
\end{proof}

From Lemma \ref{Ker(gama+alfa) integrable}, Corollary \ref{Eq def Levi flat}
and Lemma \ref{Gauge equivalence} we obtain

\begin{corollary}
\label{Eq moduli}Set 
\begin{equation*}
\mathfrak{MC}_{\delta ,\mathcal{LF}}\left( L\right) =\left\{ \left( \alpha
,S\right) \in \mathcal{Z}^{1}\left( L\right) \times \left( \Lambda
_{J}^{0,1}\left( \xi \right) \otimes \xi \right) :\ \delta a+\frac{1}{2}%
\left\{ \alpha ,\alpha \right\} =0,\ \overline{\partial }_{J}^{\alpha }S+%
\frac{1}{2}\left[ \left[ S,S\right] \right] _{\alpha }=\ -\alpha
_{J}^{0,1}\wedge H\right\} .
\end{equation*}%
Then the moduli space of deformations of Levi flat structures of $\left( \xi
,J\right) $ is 
\begin{equation*}
\mathfrak{MC}_{\mathcal{LF}}\left( \xi ,J,L\right) /\thicksim _{\mathcal{G}}
\end{equation*}%
where $\left( \alpha ,S\right) \thicksim _{\mathcal{G}}\left( \beta
,Q\right) $ if there exists $\Phi \in \mathcal{G}$ such that $\beta =\chi
\left( \Phi \right) \left( \alpha \right) $ and $Q=S_{\beta }$.
\end{corollary}

\begin{lemma}
\label{dS chi Fi}Let $Y$ be a vector field on $L$ and $\Phi ^{Y}$ the flow
of $Y$. Then%
\begin{equation*}
\frac{dS_{\chi \left( \Phi _{t}^{Y}\right) \left( 0\right) }}{dt}%
_{\left\vert t=0\right. }=-H_{Y}.
\end{equation*}%
where $H_{Y}\in \Lambda ^{0,1}\left( \xi \right) \otimes \xi $ is defined in
(\ref{HY1}).
\end{lemma}

\begin{proof}
Let $V\in \xi $. By (\ref{S chi Fi alfa}) 
\begin{eqnarray}
&&\frac{dS_{\chi \left( \Phi _{t}^{Y}\right) \left( 0\right) }}{dt}%
_{\left\vert t=0\right. }\left( V\right)  \notag \\
&=&\frac{d}{dt}_{\left\vert t=0\right. }\left( J-\omega _{\chi \left( \Phi
_{t}^{Y}\right) \left( 0\right) }^{-1}\left( \Phi _{t}^{Y}\right) _{\ast
}J\left( \Phi _{t}^{Y}\right) _{\ast }^{-1}\omega _{\chi \left( \Phi
_{t}^{Y}\right) \left( 0\right) }\right) \left( J+\omega _{\chi \left( \Phi
_{t}^{Y}\right) \left( 0\right) }^{-1}\left( \Phi _{t}^{Y}\right) _{\ast
}J\left( \Phi _{t}^{Y}\right) _{\ast }^{-1}\omega _{\chi \left( \Phi
_{t}^{Y}\right) \left( 0\right) }\right) ^{-1}  \notag \\
&=&\frac{1}{2}\left( \frac{d}{dt}_{\left\vert t=0\right. }\omega _{\chi
\left( \Phi _{t}^{Y}\right) \left( 0\right) }^{-1}\left( \Phi
_{t}^{Y}\right) _{\ast }J\left( \Phi _{t}^{Y}\right) _{\ast }^{-1}\omega
_{\chi \left( \Phi _{t}^{Y}\right) \left( 0\right) }\right) \left( JV\right)
\notag \\
&&  \label{d S chi t}
\end{eqnarray}

Since%
\begin{equation*}
\frac{d}{dt}_{\left\vert t=0\right. }\omega _{\chi \left( \Phi
_{t}^{Y}\right) \left( 0\right) }=-\frac{d}{dt}_{\left\vert t=0\right.
}\left( \chi \left( \Phi _{t}^{Y}\right) \left( 0\right) \right) X,
\end{equation*}%
Lemma \ref{d/dt(chi)=delta} gives%
\begin{equation*}
\frac{d}{dt}_{\left\vert t=0\right. }\omega _{\chi \left( \Phi
_{t}^{Y}\right) \left( 0\right) }=\delta \left( \iota _{Y}\gamma \right) X.
\end{equation*}%
It follows that%
\begin{eqnarray}
&&\left( \frac{d}{dt}_{\left\vert t=0\right. }\left( \omega _{\chi \left(
\Phi _{t}^{Y}\right) \left( 0\right) }^{-1}\left( \Phi _{t}^{Y}\right)
_{\ast }J\left( \Phi _{t}^{Y}\right) _{\ast }^{-1}\omega _{\chi \left( \Phi
_{t}^{Y}\right) \left( 0\right) }\right) \right) \left( V\right)  \notag \\
&=&\frac{d}{dt}_{\left\vert t=0\right. }\omega _{\chi \left( \Phi
_{t}^{Y}\right) \left( 0\right) }^{-1}\left( JV\right) +\frac{d}{dt}%
_{\left\vert t=0\right. }\left( \left( \Phi _{t}^{Y}\right) _{\ast }\right)
\left( JV\right) +J\left( \frac{d}{dt}_{\left\vert t=0\right. }\left( \Phi
_{t}^{Y}\right) _{\ast }^{-1}\left( V\right) +\frac{d}{dt}_{\left\vert
t=0\right. }\omega _{\chi \left( \Phi _{t}^{Y}\right) \left( 0\right)
}\left( V\right) \right)  \notag \\
&=&-\left( \delta \left( \iota _{Y}\gamma \right) \left( JV\right) \right) X-%
\mathcal{L}_{Y}\left( JV\right) +J\left( \mathcal{L}_{Y}\left( V\right)
+\left( \delta \left( \iota _{Y}\gamma \right) \left( V\right) \right)
X\right)  \notag \\
&&  \label{dSchit(V)}
\end{eqnarray}

By (\ref{B}) we have 
\begin{equation}
\delta \left( \iota _{Y}\gamma \right) \left( V\right) =\left( d\iota
_{Y}\gamma +\left( \iota _{Y}\gamma \right) \iota _{X}d\gamma -X\left( \iota
_{Y}\gamma \right) \gamma \right) \left( V\right) =V\left( \gamma \left(
Y\right) \right) +\gamma \left( Y\right) d\gamma \left( X,V\right) .
\label{delta iY(gama)}
\end{equation}%
Since $d\gamma =-\iota _{X}d\gamma \wedge \gamma $ we have

\begin{eqnarray*}
V\left( \gamma \left( Y\right) \right) &=&d\gamma \left( V,Y\right) +\gamma
\left( \left[ V,Y\right] \right) =\left( -\iota _{X}d\gamma \wedge \gamma
\right) \left( V,Y\right) +\gamma \left( \left[ V,Y\right] \right) \\
&=&-\gamma \left( Y\right) \left( \iota _{X}d\gamma \right) \left( V\right)
+\gamma \left( \left[ V,Y\right] \right) =-\gamma \left( Y\right) d\gamma
\left( X,V\right) +\gamma \left( \left[ V,Y\right] \right)
\end{eqnarray*}%
and by (\ref{delta iY(gama)} ) it follows that 
\begin{equation}
\delta \left( \iota _{Y}\gamma \right) \left( V\right) =\gamma \left( \left[
V,Y\right] \right) .  \label{delta y = ( )}
\end{equation}%
By replacing (\ref{delta y = ( )}) in (\ref{dSchit(V)}), we obtain%
\begin{eqnarray*}
&&\left( \frac{d}{dt}_{\left\vert t=0\right. }\left( \omega _{\chi \left(
\Phi _{t}^{Y}\right) \left( 0\right) }^{-1}\left( \Phi _{t}^{Y}\right)
_{\ast }^{-1}J\left( \Phi _{t}^{Y}\right) _{\ast }\omega _{\chi \left( \Phi
_{t}^{Y}\right) \left( 0\right) }\right) \right) \left( V\right) \\
&=&-\gamma \left( \left[ JV,Y\right] \right) X+\left[ JV,Y\right] +J\left( -%
\left[ V,Y\right] +\gamma \left( \left[ V,Y\right] \right) X\right)
\end{eqnarray*}%
We conclude now by (\ref{d S chi t}) and (\ref{HY1}): 
\begin{equation*}
\frac{dS_{\chi \left( \Phi _{t}^{Y}\right) \left( 0\right) }}{dt}%
_{\left\vert t=0\right. }\left( V\right) =\frac{1}{2}\left( -\left[ V,Y%
\right] +\gamma \left( \left[ V,Y\right] \right) X+J\left( -\left[ JV,Y%
\right] +\gamma \left( \left[ JV,Y\right] \right) X\right) \right) =-H_{Y}.
\end{equation*}
\end{proof}

\begin{lemma}
\label{Rel H Y H}Let $Y\in \mathcal{H}\left( L\right) $.\ Then$\ $%
\begin{equation*}
H_{Y}=\overline{\partial }\left( Y-\gamma \left( Y\right) X\right) +\gamma
\left( Y\right) H.
\end{equation*}
\end{lemma}

\begin{proof}
Let $V\in \xi .\ $We have 
\begin{eqnarray}
\left( \overline{\partial }\left( Y-\gamma \left( Y\right) X\right) \right)
\left( V\right) &=&\frac{1}{2}\left( \left[ V,Y-\gamma \left( Y\right) X%
\right] +J\left[ JV,Y-\gamma \left( Y\right) X\right] \right)  \notag \\
&=&\frac{1}{2}\left( \left[ V,Y\right] -\gamma \left( Y\right) \left[ V,X%
\right] -V\left( \gamma \left( Y\right) \right) X\right)  \notag \\
&&+\frac{1}{2}J\left( \left[ JV,Y\right] -\gamma \left( Y\right) \left[ JV,X%
\right] -\left( JV\right) \left( \gamma \left( Y\right) \right) X\right)
\label{dbar(Y-gama(Y)X}
\end{eqnarray}%
But%
\begin{equation*}
d\gamma \left( V,Y\right) =V\left( \gamma \left( Y\right) \right) -\gamma 
\left[ V,Y\right]
\end{equation*}%
and by Lemma \ref{Frobenius} 
\begin{equation*}
d\gamma \left( V,Y\right) =\left( -\iota _{X}d\gamma \wedge \gamma \right)
\left( V,Y\right) =-\gamma \left( Y\right) \left( \iota _{X}d\gamma \right)
\left( V\right) =-\gamma \left( Y\right) d\gamma \left( X,V\right) =-\gamma
\left( Y\right) \gamma \left( \left[ V,X\right] \right) ,
\end{equation*}%
so%
\begin{equation}
V\left( \gamma \left( Y\right) \right) =\gamma \left[ V,Y\right] -\gamma
\left( Y\right) \gamma \left( \left[ V,X\right] \right)  \label{V(gamaY)}
\end{equation}%
and similarly%
\begin{equation}
\left( JV\right) \left( \gamma \left( Y\right) \right) =\gamma \left( \left[
JV,Y\right] \right) -\gamma \left( Y\right) \gamma \left( \left[ JV,X\right]
\right) .  \label{JV(gamaY)}
\end{equation}%
It follows that%
\begin{equation}
\gamma \left( Y\right) H\left( V\right) =\frac{1}{2}\left( \gamma \left(
Y\right) \left[ V,X\right] -\gamma \left( Y\right) \gamma \left( \left[ V,X%
\right] \right) X+J\left( \gamma \left( Y\right) \left[ JV,X\right] -\gamma
\left( Y\right) \gamma \left( \left[ JV,X\right] \right) X\right) \right) .
\label{Gama(Y)H}
\end{equation}%
and by (\ref{V(gamaY)}) and (\ref{JV(gamaY)}) in the addition of (\ref%
{dbar(Y-gama(Y)X}) and (\ref{Gama(Y)H}) it follows that%
\begin{equation*}
\left( \overline{\partial }\left( Y-\gamma \left( Y\right) X\right) \right)
\left( V\right) +\gamma \left( Y\right) H\left( V\right) =\frac{1}{2}\left( %
\left[ V,Y\right] -\gamma \left( \left[ V,Y\right] \right) X\right) +\frac{1%
}{2}J\left( \left[ JV,Y\right] -\gamma \left( \left[ JV,Y\right] \right)
X\right) =H_{Y}\left( V\right)
\end{equation*}%
and the Lemma is proved.
\end{proof}

\begin{corollary}
\begin{equation*}
\overline{\partial }H_{Y}=\left( \delta \left( \gamma \left( Y\right)
\right) \right) ^{0,1}\wedge H
\end{equation*}
\end{corollary}

\begin{proof}
\begin{eqnarray*}
\overline{\partial }H_{Y} &=&\overline{\partial }\left( \gamma \left(
Y\right) H\right) =\overline{\partial }\left( \gamma \left( Y\right) \right)
\wedge H+\gamma \left( Y\right) \overline{\partial }H \\
&=&\overline{\partial }\left( \gamma \left( Y\right) \right) \wedge H+\gamma
\left( Y\right) \left( \iota _{X}d\gamma \right) ^{0,1}\wedge H \\
&=&\left( d\gamma \left( Y\right) +\gamma \left( Y\right) \iota _{X}d\gamma
\right) ^{0,1}\wedge H.
\end{eqnarray*}%
Since 
\begin{equation*}
\delta \left( \gamma \left( Y\right) \right) =d\gamma \left( Y\right) \gamma
\left( Y\right) +\gamma \left( Y\right) \iota _{X}d\gamma -X\left( \gamma
\left( Y\right) \right) \gamma
\end{equation*}%
it follows that 
\begin{equation*}
\delta \left( \gamma \left( Y\right) \right) =d\gamma \left( Y\right) \gamma
\left( Y\right) +\gamma \left( Y\right) \iota _{X}d\gamma
\end{equation*}%
on $\xi $ and the Corollary follows.
\end{proof}

\begin{corollary}
\label{Def Fi}Let $\beta \in \mathcal{Z}^{1}\left( L\right) $ and $\varphi
\in C^{\infty }\left( L\right) $. Then%
\begin{equation*}
\left( \beta +\delta \varphi \right) ^{0,1}\wedge H=\beta ^{0,1}\wedge H+%
\overline{\partial }\left( \varphi H\right) .
\end{equation*}%
In particular the map $\Phi _{H}:\beta \mapsto \beta ^{0,1}$ $\wedge H$
induces an application%
\begin{equation}
\Phi _{H}:H^{1}\left( \mathcal{Z}\left( L\right) ,\delta \right) \rightarrow
H^{2}\left( \Lambda ^{0,\ast }\left( \xi \right) \otimes \xi ,\overline{%
\partial }\right) .  \label{FI h}
\end{equation}
\end{corollary}

\begin{proof}
\begin{equation*}
\left( \beta +\delta \varphi \right) ^{0,1}\wedge H=\beta ^{0,1}\wedge
H+\left( \delta \varphi \right) ^{0,1}\wedge H=\beta ^{0,1}\wedge H+\left(
d\varphi +\varphi \iota _{X}d\gamma -d\varphi \left( X\right) \gamma \right)
^{0,1}\wedge H.
\end{equation*}%
Since $\gamma ^{0,1}=0$, by using Lemma \ref{dbar H} we obtain%
\begin{equation*}
\left( \beta +\delta \varphi \right) ^{0,1}\wedge H=\beta ^{0,1}\wedge H+%
\overline{\partial }\varphi \wedge H+\varphi \overline{\partial }H=\beta
^{0,1}\wedge H+\overline{\partial }\left( \varphi H\right) .
\end{equation*}
\end{proof}

\section{Moduli space of deformations of Levi flat structures and rigidity}

\begin{definition}
Let $L$ be a smooth manifold, $\left( \xi ,J\right) $ a Levi flat structure
on $L$, and $I$ an open interval in $\mathbb{R}$ containing the origin. A
deformation of the Levi flat structure $\left( \xi ,J\right) $ is a smooth
family $\left\{ \left( \xi _{t},J_{t}\right) \right\} _{t\in I}$ of Levi
flat structures on $L$ such that $\left( \xi _{0},J_{0}\right) =\left( \xi
,J\right) $.
\end{definition}

\begin{remark}
By Corollary \ref{Eq def Levi flat} a deformation of the Levi flat structure 
$\left( \xi ,J\right) $ is given by a family $\left\{ \left( \alpha
_{t},S_{\alpha _{t}}\right) \right\} _{t\in I}$, $\alpha _{t}\in \mathcal{Z}%
^{1}\left( L\right) $, $S_{\alpha _{t}}\in \Lambda _{J}^{0,1}\left( \xi
\right) \otimes \xi $ such that $\alpha _{t}$ verifies (\ref{Eq alfa}) and $%
S_{\alpha _{t}}$ verifies (\ref{Eq S}) for every $t\in I$.
\end{remark}

\begin{definition}
A $\mathfrak{MC}_{\delta ,\mathcal{LF}}\left( L\right) $-valued curve
through the origin$\ $is a smooth mapping $\lambda :\left[ -a,a\right]
\rightarrow \mathfrak{MC}_{\delta ,\mathcal{LF}}\left( L\right) $, $a>0$,
such that $\lambda \left( 0\right) =0$. We say that $\left( \beta ,P\right)
\in \mathcal{Z}^{1}\left( L\right) \times \left( \Lambda ^{0,1}\left( \xi
\right) \otimes \xi \right) $ is the tangent vector at the origin to the $%
\mathfrak{MC}_{\delta ,\mathcal{LF}}\left( L\right) $-valued curve $\lambda $
through the origin if $\left( \beta ,P\right) =\underset{t\rightarrow 0}{%
\lim }\frac{\lambda \left( t\right) }{t}=\frac{d\lambda }{dt}_{\left\vert
t=0\right. }$.
\end{definition}

\begin{theorem}
\label{Caract T moduli}Let $L$ be a smooth manifold and $\left( \xi
,J\right) $ a Levi flat structure on $L$. Let $\left\{ \left( \xi
_{t},J_{t}\right) \right\} _{t\in I}$ be a deformation of $\left( \xi
,J\right) $ given by $\left\{ \left( \alpha _{t},S_{\alpha _{t}}\right)
\right\} _{t\in I}$ , $\alpha _{t}=t\beta +o\left( t\right) $, $S_{\alpha
_{t}}=tP+o\left( t\right) $, $\alpha _{t}$,$\beta \in \mathcal{Z}^{1}\left(
L\right) $, $S_{\alpha _{t}},P\in \Lambda _{J}^{0,1}\left( \xi \right)
\otimes \xi $. Then:

1) 
\begin{eqnarray}
\delta \beta &=&0  \label{Eq delta P1} \\
\overline{\partial }P &=&-\beta ^{0,1}\wedge H.  \label{Eq delta P2}
\end{eqnarray}

2) Let $\beta ^{\prime }\in \mathcal{Z}^{1}\left( L\right) $, $P^{\prime
}\in \Lambda ^{0,1}\left( \xi \right) \otimes \xi $ such that $\beta
^{\prime }$ verifies (\ref{Eq delta P1}) and $P^{\prime }$verifies (\ref{Eq
delta P2}). Then $\left( \beta ,P\right) \thicksim _{\mathcal{G}}\left(
\beta ^{\prime },P^{\prime }\right) $ if and only if there exists $Y\in 
\mathcal{H}\left( L\right) $ such that%
\begin{eqnarray}
\beta -\beta ^{\prime } &=&\delta \iota _{Y}\gamma  \label{beta-beta'} \\
\ P-P^{\prime } &=&-H_{Y}.  \label{P-P'}
\end{eqnarray}
\end{theorem}

\begin{proof}
1) follows from (\ref{Eq alfa}) and (\ref{Eq S}) and 2) from (\ref%
{beta-beta' debut}) and Lemma \ref{dS chi Fi}.
\end{proof}

\begin{definition}
Let $\mathfrak{Z}^{p}\left( L,\xi \right) =\mathcal{Z}^{p}\left( L\right)
\oplus \Lambda ^{0,p}\left( \xi \right) \otimes \xi $, $\mathfrak{Z=}\oplus
_{p\in \mathbb{N}}\mathfrak{Z}^{p}$ and $\mathfrak{d=}\left( \mathfrak{d}%
^{p}\right) _{p\in \mathbb{N}}:\mathfrak{Z}\rightarrow \mathfrak{Z}$, where $%
\mathfrak{d}^{p}:\mathfrak{Z}^{p}\rightarrow \mathfrak{Z}^{p+1}$,%
\begin{equation}
\mathfrak{d}^{p}\left( \alpha ,P\right) =\left( \delta ^{p}\alpha ,\overline{%
\partial }P+\left( -1\right) ^{p+1}\alpha ^{0,p}\wedge H\right) .
\label{Def d fraktur}
\end{equation}
\end{definition}

\begin{proposition}
$\mathfrak{d\circ d}=0$. $\ \ $
\end{proposition}

\begin{proof}
Let $\left( \alpha ,P\right) \in \mathfrak{Z}^{p}\left( L,\xi \right) $.\ By
Corollary \ref{Forms+delta=DGLA}. we have $\delta ^{p+1}\delta ^{p}=0$ so by
Lemma \ref{dbar H} it follows that 
\begin{eqnarray*}
\mathfrak{d}^{p+1}\mathfrak{d}^{p}\left( \alpha ,P\right) &=&\mathfrak{d}%
^{p+1}\left( \delta ^{p}\alpha ,\overline{\partial }P+\left( -1\right)
^{p+1}\alpha ^{0,p}\wedge H\right) \\
&=&\left( 0,\left( \left( -1\right) ^{p+1}\overline{\partial }\left( \alpha
^{0,p}\wedge H\right) +\left( -1\right) ^{p+2}\left( \delta ^{p}\alpha
\right) ^{0,p+1}\wedge H\right) \right) \\
&=&\left( -1\right) ^{p+1}\left( 0,\overline{\partial }\alpha ^{0,p}\wedge
H+\left( -1\right) ^{p}\alpha ^{0,p}\wedge \overline{\partial }H-\left(
d\alpha +\left\{ \gamma ,\alpha \right\} \right) ^{0,p+1}\wedge H\right) \\
&=&\left( -1\right) ^{p+1}\left( 0,\overline{\partial }\alpha ^{0,p}\wedge
H+\left( -1\right) ^{p}\alpha ^{0,p}\wedge \overline{\partial }H-\overline{%
\partial }\alpha ^{0,p}\wedge H+\left\{ \gamma ,\alpha \right\}
^{0,p+1}\wedge H\right) \\
&=&\left( -1\right) ^{p+1}\left( 0,\left( -1\right) ^{p}\alpha ^{0,p}\wedge
\left( \iota _{X}d\gamma \right) ^{0,1}\wedge H-\left( \iota _{X}d\gamma
\wedge \alpha -\gamma \wedge \iota _{X}d\alpha \right) ^{0,p+1}\right)
\wedge H \\
&=&\left( -1\right) ^{p+1}\left( 0,\left( -1\right) ^{p}\alpha ^{0,p}\wedge
\left( \iota _{X}d\gamma \right) ^{0,1}-\left( \iota _{X}d\gamma \right)
^{0,1}\wedge \alpha ^{0,p}-\gamma ^{0,1}\wedge \left( \iota _{X}d\alpha
\right) ^{0,p}\right) \wedge H.
\end{eqnarray*}%
But $\gamma ^{0,1}=0$ and the Proposition is proved.
\end{proof}

\begin{notation}
We denote $\mathcal{A}$ the set of $\left( \beta ,P\right) \in \mathfrak{Z}%
^{1}\left( L,\xi \right) $ such that $\beta $ verifies (\ref{Eq delta P1})
and $P$ verifies (\ref{Eq delta P2}).
\end{notation}

\begin{theorem}
\label{Tangent moduli}Let $L$ be a smooth manifold and $\left( \xi ,J\right) 
$ a Levi flat structure on $L$. Then $\mathcal{A}/\thicksim _{\mathcal{G}}$
is canonically isomorphic with $H^{1}\left( \mathfrak{Z,d}\right) $.
\end{theorem}

\begin{proof}
Let $\left[ \left( \beta ,P\right) \right] \in \mathcal{A}/\thicksim _{%
\mathcal{G}}$, $\left( \beta ,P\right) \in \mathcal{A}$. By (\ref{Eq delta
P1}) and (\ref{Eq delta P2}), $\left( \beta ,P\right) $ defines an element $%
\widehat{\left( \beta ,P\right) }\in H^{1}\left( \mathfrak{Z,d}\right) $.

Suppose that $\left[ \left( \beta ,P\right) \right] =\left[ \left(
0,0\right) \right] $. By (\ref{beta-beta'}) and (\ref{P-P'}) there exists $%
Y\in \mathcal{H}\left( L\right) $ such that $\beta =\delta \left( \gamma
\left( Y\right) \right) $ and $P=-H_{Y}$. The Lemma \ref{Rel H Y H} implies
that 
\begin{equation*}
P=\overline{\partial }\left( Y-\gamma \left( Y\right) X\right) +\gamma
\left( Y\right) H
\end{equation*}%
so%
\begin{equation*}
\mathfrak{d}^{0}\left( \gamma \left( Y\right) ,-\left( Y-\gamma \left(
Y\right) X\right) \right) =\left( \delta \left( \gamma \left( Y\right)
\right) ,-\overline{\partial }\left( Y-\gamma \left( Y\right) X\right)
+\gamma \left( Y\right) H\right) =\left( \beta ,-H_{Y}\right) =\left( \beta
,P\right) .
\end{equation*}%
It follows that $\widehat{\left( \beta ,P\right) }$ $=$ $\widehat{\left(
0,0\right) }$ and the map $F:\mathcal{A}\rightarrow H^{1}\left( \mathfrak{Z,d%
}\right) $,$\ F\left( \left[ \left( \beta ,P\right) \right] \right) =$ $%
\widehat{\left( \beta ,P\right) }$, is well defined.

Let now $\widehat{\left( \beta ,P\right) }\in H^{1}\left( \mathfrak{Z,d}%
\right) $, with $\left( \beta ,P\right) \in \mathfrak{Z}^{1}\left( L,\xi
\right) $, $\mathfrak{d}^{1}\left( \beta ,P\right) =\left( \delta \beta ,%
\overline{\partial }P+\beta ^{0,1}\wedge H\right) =\left( 0,0\right) $. It
follows that $\beta $ verifies (\ref{Eq delta P1}) and $P$ verifies (\ref{Eq
delta P2}). In particular $\left( \beta ,P\right) \in \mathcal{A}$.

Suppose now that $\widehat{\left( \beta ,P\right) }=\widehat{\left(
0,0\right) }$, i.e. $\left( \beta ,P\right) =\mathfrak{d}^{0}\left(
0,V\right) =\left( 0,-\overline{\partial }V\right) $, $V\in \xi $. Since $%
\overline{\partial }V=H_{V}$ byRemark \ref{Rem H} ii), it follows that $%
P=-H_{V}$.\ By (\ref{beta-beta'}) and (\ref{P-P'}) it follows that $\left[
\left( \beta ,P\right) \right] =\left[ 0,0\right] \in \mathcal{A}/\thicksim $%
. In particular $F$ is an isomorphism.
\end{proof}

The Theorems \ref{Caract T moduli} and \ref{Tangent moduli} justify the
following:

\begin{definition}
The infinitesimal deformations of the Levi flat structure $\left( \xi
,J\right) $ is the collection of cohomology classes in $H^{1}\left( 
\mathfrak{Z,d}\right) $ of the tangent vectors at $0$ to $\mathfrak{MC}%
_{\delta ,\mathcal{LF}}\left( L\right) $-valued curves.We denote by $T_{%
\left[ 0\right] }\left( \mathfrak{MC}_{\delta ,\mathcal{LF}}\left( L\right)
/\thicksim _{\mathcal{G}}\right) $ the set of infinitesimal deformations of $%
\ \left( \xi ,J\right) $.
\end{definition}

\begin{definition}
Let $L$ be a smooth manifold, $\left( \xi ,J\right) $ a Levi flat structure
on $L$. We say that $\left( L,\xi ,J\right) $ is infinitesimally rigid
(respectively strongly infinitesimally rigid), if for any smooth family $%
\left\{ \left( \alpha _{t},S_{\alpha _{t}}\right) \right\} _{t\in I}$
defining a deformation of the Levi flat structure $\left( \xi ,J\right) $,
the class of the tangent vector to the $\mathfrak{MC}_{\delta ,\mathcal{LF}%
}\left( L\right) $-valued curve $t\mapsto \left( \alpha _{t},S_{\alpha
_{t}}\right) $ in $\mathcal{A}/\thicksim _{\mathcal{G}}$vanishes
(respectively the tangent vector to the $\mathfrak{MC}_{\delta ,\mathcal{LF}%
}\left( L\right) $-valued curve $t\mapsto \left( \alpha _{t},S_{\alpha
_{t}}\right) $ vanishes).
\end{definition}

\begin{corollary}
Let $L$ be a smooth manifold $\left( \xi ,J\right) $ a Levi flat structure
on $L$ such that $H^{1}\left( \mathfrak{Z,d}\right) =0$. Then $\left( L,\xi
,J\right) $ is infinitesimally rigid.
\end{corollary}

\begin{corollary}
Let $L$ be a smooth manifold and $\left( \xi ,J\right) $ an exact Levi flat
structure on $L$. Then $\mathcal{A}/\thicksim _{\mathcal{G}}$ is canonically
isomorphic to $H^{1}\left( \mathcal{Z}^{\ast }\left( L\right) ,\delta
\right) \times H^{1}\left( \Lambda _{J}^{0,\ast }\left( \xi \right) \otimes
\xi \right) $ and we have a canonical injection \ 
\begin{equation*}
T_{\left[ 0\right] }\left( \mathfrak{MC}_{\delta ,\mathcal{LF}}\left(
L\right) /\thicksim \right) \hookrightarrow H^{1}\left( \mathcal{Z}^{\ast
}\left( L\right) ,\delta \right) \times H^{1}\left( \Lambda _{J}^{0,\ast
}\left( \xi \right) \otimes \xi ,\overline{\partial }_{J}\right) .
\end{equation*}
\end{corollary}

\begin{proof}
By Theorem \ref{Tangent moduli} $\ \mathcal{A}/\thicksim _{\mathcal{G}}$ is
canonically isomorphic to $H^{1}\left( \mathfrak{Z,d}\right) $. There exists
a defining couple $\left( \gamma ,X\right) $ such that the $\left(
0,1\right) $-form $H$ associated to $\left( \gamma ,X\right) $ vanishes. It
follows that $\mathfrak{d}=\delta \oplus \overline{\partial }_{J}$ and so%
\begin{equation*}
H^{1}\left( \mathfrak{Z,d}\right) =H^{1}\left( \mathcal{Z}^{\ast }\left(
L\right) ,\delta \right) \times H^{1}\left( \Lambda _{J}^{0,\ast }\left( \xi
\right) \otimes \xi ,\overline{\partial }_{J}\right) .
\end{equation*}
\end{proof}

\begin{example}
Let $L=\mathbb{CP}_{n}\times S^{1}$with its exact Levi flat structure given
by the DGLA\ defining couple $\left( \gamma ,X\right) =\left( dt,\frac{%
\partial }{\partial t}\right) $, where $t$ runs in $S^{1}$. Since $\delta
=d_{b}$ and $\mathbb{CP}_{n}$ is simply connected and infinitesimally rigid,
it follows that $H^{1}\left( \mathfrak{Z,d}\right) =0$.
\end{example}

\renewcommand\baselinestretch{1} 
\bibliographystyle{amsplain}
\bibliography{defintr,deformations}

\end{document}